\documentclass{amsproc}
\usepackage{tikz}
\usepackage[colorlinks, citecolor=blue]{hyperref}

\newtheorem{theorem}{Theorem}[section]
\newtheorem{lemma}[theorem]{Lemma}

\theoremstyle{definition}
\newtheorem{definition}[theorem]{Definition}
\newtheorem{example}[theorem]{Example}

\newtheorem{claim}[theorem]{Claim}
\newtheorem{question}[theorem]{Question}
\newtheorem{corollary}[theorem]{Corollary}
\newtheorem{assumption}[theorem]{Assumption}
\newtheorem{proposition}[theorem]{Proposition}

\theoremstyle{remark}
\newtheorem{remark}[theorem]{Remark}

\numberwithin{equation}{section}

\newcommand{\Ev}{\mathbb{E}}
\newcommand{\E}{\mathbf{E}}

\newcommand{\h}{\mathbf{H}}

\newcommand{\N}{\mathbf{N}}
\newcommand{\Z}{\mathbf{Z}}

\newcommand{\Pv}{\mathbb{P}}

\newcommand{\R}{\mathbf{R}}

\newcommand{\CF}{\mathcal {F}}
\newcommand{\CG}{G}

\newcommand{\CS}{\mathcal {S}}
\newcommand{\CT}{\mathcal {T}}

\newcommand{\CX}{\mathcal {X}}

\newcommand{\diam}{{\rm diam}}

\newcommand{\one}{\mathbf{1}}

\newcommand{\Hh}{H}

\newcommand*{\pid}{\pi^{\diamond}}

\newcommand*{\la}{\lambda}

\newcommand*{\ve}{\varepsilon}

\newcommand*{\un}[1]{\underline{#1}}

\newcommand*{\be}{\begin{equation}}
\newcommand*{\ee}{\end{equation}}
\newcommand*{\ba}{\begin{aligned}}
\newcommand*{\ea}{\end{aligned}}
\newcommand*{\barr}{\begin{array}{c}}
\newcommand*{\earr}{\end{array}}

\newcommand*{\Cov}{{\text{\bf Cov}}}
\newcommand*{\fl}[1]{\lfloor{#1}\rfloor}

\newcommand*{\trel}{t_{{\rm rel}}}

\newcommand*{\tcov}{t_{{\rm cov}}}
\newcommand*{\thit}{t_{{\rm hit}}}
\newcommand*{\tmix}{t_{{\rm mix}}}

\newcommand*{\TV}{{\rm TV}}

\renewcommand{\diam}{{\rm diam}}

\renewcommand{\and}{\hbox{ {\rm and} }}
\newcommand{\ldist}{\rho}

\newcommand{\diag}{\operatorname{diag}}
\newcommand{\tr}{\operatorname{tr}}
\newcommand*{\sss}{\scriptscriptstyle}

\begin{document}

\title{Topics in Markov chains: mixing and escape rate}

\author{J{\'u}lia Komj{\'a}thy}
\address{Department of Stochastics, Budapest University of Technology and Economics, Budapest, Hungary}
\curraddr{Eindhoven University of Technology, Eindhoven, The Netherlands}

\email{j.komjathy@tue.nl}
\thanks{J. Komj\'athy was supported by the grant
 KTIA-OTKA  $\#$ CNK 77778, funded by the Hungarian National Development Agency (NF\"U)  from a
source provided by KTIA}

\author{Yuval Peres}
\address{Microsoft Research}
\email{peres@microsoft.com}

\subjclass{Primary 60J10, 160D05, 37A25}
\date{September 25, 2013 and, in revised form, January 16, 2015.}

\keywords{Random walk, generalized lamplighter walk, wreath product, mixing time, relaxation time, Varopolous-Carne long range estimates.}

\begin{abstract}
These are the notes for the minicourse on Markov chains delivered at the Saint Petersburg Summer School, June 2012. The main emphasis is on methods for estimating mixing times (for finite chains) and escape rates (for infinite chains).
Lamplighter groups are key examples in both topics and the Varopolous-Carne long range estimate is useful in both settings.
\end{abstract}

\maketitle

\section{Preliminaries}
We start with preliminary notions necessary for the analysis of mixing and relaxation time of Markov chains. For much more on this topic see the books \cite{AF02, LPW08}.
\subsection{Total variation distance and coupling}
We start with the definition of total variation distance and coupling of two probability measures:
\begin{definition}\label{def::tv}
Let $\CS$ be a state space, and $\mu$ and $\nu$ be two probability measures defined on $\CS$.
Then the {\bf total variation distance} between  $\mu$ and $\nu$ is defined as
\[ \|\mu - \nu\|_{\TV} = \max_{A\subset \CS} |\mu(A) - \nu(A)|. \]
\end{definition}

\begin{definition}
A {\bf coupling} of two probability measures $\mu$ and $\nu$ on $\CS$ is a pair of random variables $(X,Y)$ having joint distribution $q$ on $\CS \times \CS$ such that the marginal distributions are $\Pv[X=x]=\sum_{y\in \CS} q(x,y)=\mu(x)$ and $\Pv[Y=y]=\sum_{x\in \CS}q(x,y)=\nu(y)$ for every $x,y\in \CS.$
\end{definition}
Then, the followings give equivalent characterizations of the total variation distance $\|\mu - \nu\|_{\TV}$:
\begin{align}  &  \max_{A\subset \CS} |\mu(A) - \nu(A)|\label{eq::tv1}\\
& \frac12 \|\mu-\nu\|_1= \frac12 \sum_{x\in \CS} |\mu(x)-\nu(x)| \label{eq::tv2}\\
& \sum_{x\in \CS: \mu(x)>\nu(x)} \big(\mu(x)-\nu(x)\big)\label{eq::tv3} \\
& \inf\left\{\Pv[X \neq Y] : (X, Y) \text{ is a coupling of }\mu \text{ and }\nu \right\} \label{eq::tv4}
\end{align}

\begin{proof}
It is intuitively clear that the set $B:= \{x: \mu(x)\ge \nu(x) \}$ or its complement maximizes the right hand side in Definition \ref{def::tv}. To give a formal proof, take $A\subset \CS$.
 From the definition of $B$ it follows that
\be\label{eq::tvproof1} \mu(A)-\nu(A)\le \mu(A \cap B) - \nu(A \cap B) \le \mu(B)-\nu(B). \ee
This proves that $\eqref{eq::tv1}\le \eqref{eq::tv3}$. But, if we take $A=B$, then the maximum is taken, i.e. $\eqref{eq::tv1}= \eqref{eq::tv3}.$
By the same reasoning, with $B^c:=\CS \setminus B$ we also have
\be\label{eq::tvproof2} \nu(A)-\mu(A)\le \nu(A \cap B^c) - \mu(A \cap B^c) \le \nu(B^c)-\mu(B^c). \ee
Note that since $\mu(B^c)=1-\mu(B),\  \nu(B^c)=1-\nu(B)$ the right hand side of \eqref{eq::tvproof1} and $\eqref{eq::tvproof2}$ coincide, thus yielding
\[  \max_{A\subset \CS} |\mu(A) - \nu(A)| = \frac12 \left(\mu(B)-\nu(B) + \nu(B^c)-\mu(B^c) \right) = \frac12 \sum_{x\in \CS} |\mu(x)-\nu(x)|, \]
proving \eqref{eq::tv1}=\eqref{eq::tv2}.

To see that \eqref{eq::tv1} $\le$ \eqref{eq::tv4}, we write
\[ \ba \mu(A)-\nu(A) &= \Pv[X\in A] - \Pv[Y\in A] \\
&\le \Pv[X\in A, Y \notin A] \\
&\le \Pv[X \neq Y]. \ea \]
For the other direction we construct a coupling for which the infimum is attained. Intuitively, what we do is pack as much mass into the diagonal $q(x,x)$ as we can, such that we still maintain the correct marginal measures.
More formally, let us define
\[ \ba q(x,x)&:= \min\{\mu(x), \nu(x)\}\\
q(x,y)&:= 0 \mbox{ if } q(x,x)=\mu(x) \mbox{ or } q(y,y)=\nu(y)\\
q(x,y)&=\frac{(\mu(x)-\nu(x))(\nu(y)-\mu(y))}{1-\sum_z q(z,z)} \mbox{ if } q(x,x)=\nu(x) \mbox{ and } q(y,y)=\mu(y).
 \ea\]
 Intuitively, we put the maximal possible weight in the diagonal of $q$, (which is $\min\{ \mu(x), \nu(x)\}$ and then we put zeros in the corresponding column or row, depending on the minimum being $\mu(x)$ or $\nu(x)$. Finally, we fill the rest out with conditionally independent choice, i.e. on $B\times B^c$ we distribute $(\mu(x)-\nu(x))\cdot (\nu(y)-\mu(y))>0$ with the normalizing factor  $1-\sum_z q(z,z)$. Mind that this is not the only way of doing the coupling. To check that the marginals are correct is left to the reader.
With this particular coupling, \eqref{eq::tv4} becomes
\[ \ba \eqref{eq::tv4}\le \Pv(X\neq Y) &= 1-\sum_x q(x,x) = 1- \sum_x\min\{ \mu(x),\nu(x) \} \\
&=\sum_x \mu(x) -\left(\sum_{x: \mu(x)>\nu(x)} \nu(x) + \sum_{x: \mu(x) \leq \nu(x)} \mu(x)\right)\\
&=\sum_{x: \mu(x)>\nu(x)} \left[\mu(x)-\nu(x)\right]=\eqref{eq::tv3}.
\ea \]
With this we have \eqref{eq::tv4}$\le$ \eqref{eq::tv3}=\eqref{eq::tv1}, finishing the proof.
\end{proof}

\subsection{Mixing in total variation distance}
Let $X_t$ be a Markov chain on state space $\CS$ with transition matrix $P$, and stationary measure $\pi$ on $\CS$. That is, $\pi P = \pi$.
If $P$ is irreducible and aperiodic, then the measure $\mu_t(y)=P^t(x,y)$ is converging to the stationary measure exponentially fast, i.e. there exists an $\alpha\in (0,1)$ such that
\[ \|P^t(x,.)-\pi(.) \|_{TV} \le C \alpha^t. \]
These asymptotics hold for a single chain as the time $t$ tends to infinity. However, we are rather interested in the finite time behavior of a sequence of Markov chains, i.e. how long one has to run the Markov chain as a function of $|\CS|$, to get $\ve$-close to stationary measure, for fixed $\ve$.

Thus, let us define
\be\label{def::dxt}  d_x(t):= \| P^t(x,\cdot)-\pi(\cdot)\|_{\text{TV}}; \quad d(t):= \max_{x\in \CS} d_x(t). \ee
Then, the $\ve$-mixing time of a Markov Chain on a graph $G$ is defined as
  \begin{equation}
\label{eqn::tmix_definition}
 \tmix(\CG,\ve) := \min\left\{ t \geq 0: d(t) \leq \ve \right\}.
\end{equation}
Throughout, we set $\tmix(\CG) := \tmix(\CG, \tfrac{1}{4})$.
The characterisation \eqref{eq::tv4} suggests that sometimes it is more convenient to work with chains started from two different initial states, so let us define
\[ \bar d(t) := \max_{x,y\in \CS} \| P^t(x,\cdot)-P^t(y,\cdot)\|_{\text{TV}}.\]
Then, we have the following comparison:
\begin{lemma}\label{lem::dtdt} With the above definitions,
\be\label{eq::dtdbart} d(t)\le  \bar d(t)\le 2 d(t)\ee
Further, the function $\bar d(t)$ is submultiplicative, i.e.
\be \label{eq::dbarsubmulti}\bar d(t+s) \le \bar d(t) \bar d(s), \ee
and combining yields
\be\label{eq::dsubmulti} d(k t) \le 2^k d(t)^k\ee
\end{lemma}
\begin{proof}
We only prove \eqref{eq::dtdbart} here. The proof of \eqref{eq::dbarsubmulti} is the proof of  Lemma 4.12 in \cite{LPW08}, and \eqref{eq::dsubmulti} is an easy combination of the first two statements of the lemma.  To prove the second inequality in \eqref{eq::dtdbart}, we use the triangle inequality
\[ \ba \bar d(t) = \| P^t(x,\cdot)-P^t(y,\cdot)\|_{\text{TV}} &\le  \| P^t(x,\cdot)-\pi(\cdot)\|_{\text{TV}}+ \| \pi(\cdot)-P^t(y,\cdot)\|_{\text{TV}}\\
&\le 2 d(t),\ea \]
and for the first inequality we can use that $\pi P^t =\pi$ to get
\[ d_x(t)= \max_{A\in \CS} |P^t(x,A)-\pi(A)| \\
= \max_{A\in \CS} \Big|\sum_{y\in \CS} \pi(y)P^t(x,A)-P^t(y,A)\Big|. \]
 Now, by the triangle inequality the right hand side is at most
\[ \ba \max_{A\in \CS}\sum_{y\in \CS} \pi(y)\left|P^t(x,A)-P^t(y,A)\right| &\le \sum_{y\in \CS} \pi(y)\max_{A\in \CS}\left|P^t(x,A)-P^t(y,A)\right|\\
&= \bar d(t). \ea \]
\end{proof}
The definition $\bar d(t)$ is extremely useful, since it allows us to relate the mixing time of the chain to the tail behavior of the so called {\bf coupling time}:
Given a coupling $(X_t,Y_t)$ of $P^t(x,\cdot)$ and $P^t(y,\cdot)$, let us define
\[ \tau_{\mathrm{couple}}:= \min\{t: X_t = Y_t\}. \]
Then we have
\be\label{eq::couple} d(t)\le \bar d(t)\le \max_{x,y} \Pv[X_t\neq Y_t] = \max_{x,y} \Pv[ \tau_{\mathrm{couple}}>t]. \ee

With all these prerequisites in our hands, we can state and prove our first theorem:
\begin{theorem}\label{thm::cyclemix}
The mixing time of $C_n$, the cycle on $n$ vertices is bounded from above by
\[ \tmix(C_n)\le n^2. \]
\end{theorem}
\begin{proof}
We will construct a coupling of the measures $P^t(x,\cdot)$ and $P^t(y,\cdot)$ and use \eqref{eq::couple} to estimate $d(t)$.
Note that $P^t(x,\cdot)$ and $P^t(y,\cdot)$ are the transition measures of two lazy random walks, say $X_t$ and $Y_t$, with $X_0=x$ and $Y_0=y$. Thus, we construct a coupling of $(X_t,Y_t)$ as follows: we couple the increments of the walks, as long as $X_t\neq Y_t$ holds:
\[ \ba \Pv( X_t\!-\!X_{t-1}\!=\!0, Y_t\!-\!Y_{t-1}\!=\!+1)\! =\!\tfrac{1}{4}; &\  \ \Pv( X_t\!-\!X_{t-1}\!=\!0, Y_t\!-\!Y_{t-1}\!=\!-1)\! =\!\tfrac14;\\
\Pv( X_t\!-\!X_{t-1}\!=\!+1, Y_t\!-\!Y_{t-1}\!=\!0) \!=\! \tfrac14; &\ \  \Pv( X_t\!-\!X_{t-1}\!=\!-1, Y_t\!-\!Y_{t-1}\!=\!0)\! =\!\tfrac14.
\ea \]
If the two walks meet than they stay together from that point on.
It is easy to check that the marginals of the two walks are correct. The advantage of this coupling is that before collision the two walks never move at the same time. I.e., the clockwise distance $D_t=X_t-Y_t$ changes at each step by $+1$ or $-1$. This means that $D_t$ is doing a simple (non-lazy) symmetric random walk on $\{0,1,\dots n\}$ with $D_0:=k\in \{1,\dots n-1 \}$, and we are waiting until it hits $0$ or $n$. This is exactly the well known \emph{Gambler's ruin} problem. The coupling time is then $\tau_{0,n}$, the hitting time of the set $\{0,n\}$. We can use the martingale $D_t$ and use optional stopping to calculate its expected value:
\[ k=\Ev_k[D_0]=\Ev_k[D_{\tau_{0,n}}]=\Pv_k[D_{\tau_{0,n}}=n] n,\]
from which $\Pv_k[D_{\tau_{0,n}}=n]=k/n$. Then, $D_t^2-t$ is also a martingale, (to check is left for the reader as an exercise)
and using the previous calculation and optional stopping gives \[ \Ev_k[\tau_{0,n}]= k (n-k)\le n^2 /4. \]

Combining Lemma \ref{lem::dtdt}, the characterisation \eqref{eq::tv4} of the total variation distance and the previous calculations with a Markov's inequality, we arrive at the following sequence of inequalities:
\[ d(t)\le \bar d(t)\le \max_{x,y} \Pv[X_t\neq Y_t]= \max_{k}\Pv_k[D_t > t] \le \frac{\Ev_k[D_t]}{t} \le \frac{n^2}{4t}. \]
Now let us set $t=n^2$, then we get $d(n^2)\le1/4$, implying $\tmix(C_n)\le n^2$.
\end{proof}
A similar coupling can be used to give an upper bound on the mixing time on the $d$-dimensional tori:
\begin{theorem}\label{thm::torusmix}
The total variation mixing time on $\Z_n^d$, the $d$-dimensional torus is bounded from above by
\be \tmix(\Z_n^d) \le 3 d \log d\  n^2. \ee
\end{theorem}
\begin{proof}
We couple the two walks $\un{X}_t=(X_t^1, X_t^2, \dots X_t^d)$ and $\un{Y}_t=(Y_t^1, Y_t^2, \dots Y_t^d)$ coordinate-wise with the same coupling as in the proof of Theorem \ref{thm::cyclemix}. More precisely, at each step we first pick a uniform number $U_t\in \{1,2, \dots d\}$ independently of everything else, and then, we check if the corresponding coordinates $X_t^{U_t}, Y_t^{U_t}$ coincide or not. If so, we move both walks with the same increment: $0$, $+1$ or $-1$ with probabilities $1/2,1/4,1/4$ each. If $X_t^{U_t}\neq Y_t^{U_t}$, then we apply the coupling described in the proof of Theorem \ref{thm::cyclemix} for the $U_t$th coordinate.
Let $D_t^i$ denote the clockwise difference between $X_t^i$ and $Y_t^i$, and $\tau_i$ denote the first time when $D_t^i$ hits $\{0,n\}$.
Since each coordinate $i$ has a Geometric($1/d$) waiting time for its next move, the marginal distribution of each $\tau_i$ can be written as
\[ \tau_i = \sum_{j=1}^{\tau_{0,n}^{(i)}} Z_j \]
 with $Z_j\sim \mathrm{Geo}(1/d)$, and $\tau_{0,n}^{(i)}\sim \tau_{0,n}$ as in the proof of Theorem \ref{thm::cyclemix}.
This gives that  $\Ev[\tau_i] \le \frac{d n^2}{4} $. Note that this bound holds for every starting point $x_i, y_i$. So we can run the chain in blocks of $d n^2/2$ and then in each block we hit the set $\{ 0,n\}$ with probability at least $1/2$ by Markov's inequality. Hence the hitting of the set $\{0,n\}$ is stochastically dominated by a random variable of the form $\frac12 d n^2 \mathrm{Geo}(1/2)$. This yields the bound
\[ \Pv[\tau_i>t ] \le 2 \left(\frac{1}{2}\right)^{\frac{2 t}{d n^2}},\]
where the factor $2$ comes from ignoring the integer part of $\frac{2 t}{d n^2}$. Set $t= 3 d \log d \cdot n^2$,
then, for all $d\ge 2$:
\[ \Pv[\un X_t \neq \un Y_y]= \Pv[\exists i: \tau_i>t]\le d \cdot \Pv[\tau_i > t]\le 2d \left(\frac{1}{2}\right)^{\frac{2t}{d n^2}}= 2 d^{1-6 \log 2}\le \frac14. \]
Hence we have $\tmix(\Z_n^d)\leq 3 d \log d n^2$, finishing the proof.  \end{proof}

\subsection{Strong stationary times}
In many cases the following random times give a useful bound on mixing times:
\begin{definition}\label{def::sst}
A randomized stopping time $\tau$ is called a {\emph strong stationary time} for the Markov chain $X_t$ on $G$ if
\be\label{eq::defsst} \Pv_{x}\left[ X_{\tau}=y, \tau = t \right] = \pi(y) \Pv_x[\tau=t], \ee
that is, the position of the walk when it stops at $\tau$ is distributed according to $\pi$ and independent of the value of $\tau$.
\end{definition}
\noindent The adjective randomized means that the stopping time can depend on some extra randomness, not just purely the trajectories of the Markov chain, for a precise definition see \cite[Section 6.2.2]{LPW08}.
\begin{definition}\label{def::halting_state} A state $h(x)\in V(G)$ is called a  \emph{halting state} for a stopping time $\tau$ and initial state $x$ if $\{X_t=h(x)\}$ implies $\{\tau\le t\}$.
\end{definition}

Strong stationary times are useful since they are closely related to an other notion of distance from the stationary measure. We define
\begin{definition}
The \emph{separation distance} $s(t)$ is defined as
\be\label{eq::sep} s(t):=\max_{x\in \CS} s_x(t) \mbox{ with } s_x(t):= \max_{y\in \CS} \left( 1- \frac{P^t(x,y)}{\pi(y)}\right).\ee
\end{definition}
We mention that the separation distance is not a metric.

 The relation between the separation distance and any strong stationary  time $\tau$ is the following inequality from \cite{AF02} or \cite[Lemma 6.11]{LPW08}:
\be\label{eq::sep_ineq}\forall x \in \CS:  s_x(t)\le \Pv_x(\tau>t). \ee
The proof is just two lines, so we include it here for the reader's convenience:
for any $y$ we have
\be \label{eq::proofofsst} 1-\frac{P^t(x,y)}{\pi(y)} \le 1-\frac{\Pv_x[X_t=y, \tau\le t]}{\pi(y)} \ee
Now \eqref{eq::defsst} implies that the last expression equals
\[  1-\frac{\pi(y)\Pv_x[ \tau\le t]}{\pi(y)}=\Pv_x[\tau>t].\]

Later we will need a slightly stronger result than \eqref{eq::sep_ineq}, namely from \eqref{eq::proofofsst} it follows that \emph{if $\tau$ has a halting state $h(x)$ for $x$}, then putting $y=h(x)$ yields that equality holds in \eqref{eq::sep_ineq}.
Unfortunately, the statement can not be reversed: the state $h(x,t)$ maximizing the separation distance at time $t$ can also depend on $t$ and thus the existence of a halting state is not necessarily needed to get equality in \eqref{eq::sep_ineq}.

On the other hand, one can always construct $\tau$ such that \eqref{eq::sep_ineq} holds with equality for every $x\in \CS$. This $\tau$ does not necessarily obeys halting states.  This is one of the main ingredients to our proofs in Section \ref{s:lamp}, so we cite it as a Theorem (with adjusted notation).
\begin{theorem}\label{thm::AD}[Aldous, Diaconis]\emph{\cite[Proposition 3.2]{AD86}}
Let $(X_t, t\ge 0)$ be an irreducible aperiodic Markov chain on a finite state space $\CS$ with initial state $x$ and stationary distribution $\pi$, and let $s_x(t)$ be the separation distance defined as in \eqref{eq::sep}.
Then
\begin{enumerate}
\item if $\tau$ is a strong stationary time for $X_t$, then $s_x(t) \le \Pv_x(\tau>t)$ for all $t\ge 0.$
\item Conversely, there exists a strong stationary time $\tau$ such that $s_x(t)= \Pv_x(\tau>t)$ holds with equality.
    \end{enumerate}
\end{theorem}
Combining these, we will call a strong stationary time $\tau$ \emph{separation optimal} if it achieves equality in \eqref{eq::sep_ineq}. Mind that every stopping time possessing halting states is separation optimal, but not the other way round.

The next lemma relates the total and the separation distance:
\begin{lemma}\label{lem::sep_tv}
For any reversible Markov chain and any state $x \in \CS$, the separation distance from initial vertex $x$ satisfies:
\begin{align}
d_x(t) &\le s_x(t) \label{eq::tvlesep}\\
s_x(2t) & \le 4 d(t) \label{eq::sepletv}
\end{align}
\end{lemma}
\begin{proof}
For a short proof of \eqref{eq::tvlesep} see \cite{AF02} or \cite[Lemma 6.13]{LPW08}, and combine \cite[Lemma 19.3]{LPW08} with a triangle inequality to conclude \eqref{eq::sepletv}. Here we write the proofs for the reader's convenience. We have
\[ \ba d_x(t)= \sum_{\begin{subarray}{c}y\in\CS\\  P^t(x,y)< \pi(y)\end{subarray}}  \left[ \pi(y)- P^t(x,y)\right] &= \sum_ {\begin{subarray}{c}y\in\CS\\  P^t(x,y)< \pi(y)\end{subarray}} \pi(y)\left[1-\frac{P^{t}(x,y)}{\pi(y)}\right]\\
&\le \max_{y} \left[1-\frac{P^{t}(x,y)}{\pi(y)}\right] = s_x(t).
\ea \]
To see \eqref{eq::sepletv}, we mind that reversibility means that $P^t(z,y)/\pi(y)=P^t(y,z)/\pi(z)$. Hence we have
\[ \frac{P^{2t}(x,y)}{\pi(y)} = \sum_{z\in \CS} \frac{P^t(x,z)P^t(z,y)}{\pi(y)}= \sum_{z\in \CS} \frac{P^t(x,z)P^t(y,z)}{\pi(z)}\cdot \sum_{z\in \CS}\pi(z)  \]
Applying Cauchy-Schwarz to the right hand side implies
\[ \frac{P^{2t}(x,y)}{\pi(y)} \ge \left( \sum_{z\in \CS} \sqrt{P^t(x,z)P^t(y,z)}\right)^2 \ge \left(  \sum_{z\in \CS} P^t(x,z) \wedge P^t(y,z)\right)^2\!\!. \]
Recall \eqref{eq::tv4}, i.e.
\[ \sum_{z} \mu(z) \wedge \nu(z) = 1-\| \mu- \nu \|_{\text{TV}}. \]
Combining this with the previous calculation results in
\[ 1- \frac{P^{2t}(x,y)}{\pi(y)} \le 1- \left(1-\| P^t(x,.), P^t(y,.) \|_{\text{TV}}\right)^2. \]
Using the triangle inequality $\| P^t(x,.)- P^t(y,.) \|_{\text{TV}} \le 2 d(t)$ and expanding the terms yields \eqref{eq::sepletv}.
\end{proof}
We demonstrate the use of strong stationary times by analysing the \emph{separation time} of the $d$-dimensional hypercube: the separation time is defined similarly as the mixing time in \eqref{eqn::tmix_definition} by replacing $d(t)$ by $s(t)$.
\begin{theorem}
For the lazy random walk on the hypercube $H_d=\{0,1\}^d$,
\[ t_{\mathrm{sep}}(H_d,\ve) \le d \log d + \log (1/\ve) d. \]
\end{theorem}
\begin{proof}
We construct the following strong stationary time for the lazy random walk on the hypercube: independently in each step, we pick a uniform coordinate $U_t\in \{0,1,\dots d\}$, and then independently of the current values and everything else, we set $X_t^{U_t}=1$ with probability $1/2$ and $X_t^{U_t}=0$ with probability $1/2$. By doing so, the probability that the chain stays put is exactly $1/2$, and with probability $1/2$ it moves to a position chosen uniformly among all neighboring vertices, i.e., we get exactly the transition probabilities for a lazy random walk on the hypercube.

Define $\tau_{\text{refresh}}$ as the first time that all coordinates have been chosen.
 Then, at $\tau_{\text{refresh}}$, each coordinate $i \in \{1, \dots d\}$ has been selected already at least once, thus, its position is $0$ or $1$ with probability $1/2$ each, independently of how long we had to wait for $\tau_{\mathrm{refresh}}$ to happen. Also, if the original state was $\un x=(x_1, x_2, \dots x_d)$, then to reach  $h(\un x)= (1-x_1, 1-x_2, \dots 1-x_d)$, we have to refresh each  coordinate at least once, i.e., $h(\un x)$ is a halting state for $\tau_{\text{refresh}}$. This shows that $\tau_{\mathrm{refresh}}$ is a \emph{separation-optimal strong stationary time} for the lazy RW on the hypercube.

Note that the distribution of $\tau_{\mathrm{refresh}}$ is the same as that of the \emph{coupon collector} problem:
\[ s_x(t)=\Pv_x[\tau_{\text{refresh}} >t] = \Pv[\exists i\in \{1,\dots,d\}: \forall s\le t\    U_s \neq i]\le d \left(1-\frac1d\right)^t.\]
By putting $t = d \log d - \log (\ve) d$, the right hand side of the previous display is less than $e^{\log(\ve)}=\ve $, finishing the proof.
\end{proof}
\begin{remark}\normalfont It is known (see \cite[Example 12.17]{LPW08} that the \emph{total variation mixing time} of the hypercube is at $\frac12 d \log d$, hence we have a factor $2$ between the separation and tv-mixing time on $H_d$. Comparing it to the estimate in \eqref{eq::sepletv}, this shows that the factor $2$ there can be sharp.
\end{remark}

The following lemma will be used later to determine the spectral gap of the lamplighter chain: (\cite[Corollary 12.6]{LPW08})
\begin{lemma}\label{lem::dtlimlambda}
For a reversible, irreducible and aperiodic Markov chain,
\be\ba\label{eq::la_lim_upper} d_x(t) &\le s_x(t) \le  \frac{\la_*^t}{\pi_{\min}},\\
|\la_2|^t&\le 2 d(t)\ea \ee
with $\pi_{\min}= \min_{y\in \CS} \pi(y)$  and $\la_*= \max\{|\la|: \la \text{ eigenvalue of P, } \la \neq 1\}.$
As a consequence we have\[ \lim_{t\to \infty} d(t)^{1/t}= \la_*.\]
\end{lemma}
\begin{proof}
 Follows from \cite[Equation (12.11), (12.13)]{LPW08}.We note that Lemma \ref{lem::sep_tv} implies that the assertion of Lemma \ref{lem::dtlimlambda} stays valid if we replace $d(t)^{1/t}$ by the separation distance $s(t)^{1/t}$.
 \end{proof}

\section{Mixing times of lamplighter graphs}\label{s:lamp}

In this section, we will use the preliminaries from the previous sections to determine the mixing and relaxation time of the random walk on lamplighter graphs. The intuitive representation of the walk is the following: a lamplighter moves according to a simple random walk on the vertices of a \emph{base graph} $G$. Further, there is an identical lamp attached to each vertex $v\in G$, and each of the lamps is either on or off. We denote the state of the lamp at vertex $v\in G$ by $f_v$. Then, as the lamplighter walks along the base graph, he switches on or off lamps on its path randomly. More precisely, we are analyzing the following dynamics below: one move of the lamplighter walk corresponds to three elementary steps: he randomizes the lamp on its current position, then he moves according to a lazy simple random walk on the base graph, then he randomizes the lamp at its arrival position.

Suppose that $\CG$ is a finite connected graph with vertices $V(\CG)$ and edges $E(\CG)$.  We refer to $\CG$ as the base graph. Let $\CX(\CG) = \{\un f \colon V(\CG) \to \{0,1\}\}$ be the set of markings of $V(\CG)$ by elements of $\{0,1\}$.
 The wreath product $\Z_2 \wr \CG$ is the graph whose vertices are pairs $(\un f,x)$
where $\un f=\left(f_v\right)_{v\in V(\CG)} \in \CX(\CG)$ and $x \in V(\CG)$.
There is an edge between $(\un f,x)$ and $(\un g,y)$ if and only if $(x,y) \in E(\CG)$, and $f_z = g_z$ for all $z \notin \{x,y\}$.  Suppose that $P$ is the transition matrix for lazy random walk on $\CG$.  The lamplighter walk $X^\diamond$ is the Markov chain on $\Z_2\wr \CG$ which moves from a configuration $(\un f,x)$ by
\begin{enumerate}
\item picking $y$ adjacent to $x$ in $\CG$ according to $P$, then
\item updating each of the values of $f_x$ and $f_y$ independently to a uniform random value in $\{0,1\}$.
\end{enumerate}
The state of lamps $f_z$ at all other vertices $z\in \CG$ remain fixed.  It is easy to see that with stationary distribution $\pi_\CG$ for the random walk on $G$, the unique stationary distribution of $X^\diamond$ is the product measure
\[ \pid\big((\un f,x)\big)= \pi_\CG(x) \cdot 2^{-|G|},\]
and $X^\diamond$ is itself reversible.  In this notes, we will be concerned with the special case that $P$ is the transition matrix for the \emph{lazy random walk} on $\CG$.  In particular, $P$ is given by
\begin{equation}
\label{eq::lazy_rw_definition}
P(x,y) := \begin{cases} \frac{1}{2} \text{ if } x = y,\\ \frac{1}{2d(x)} \text{ if } \{x,y\} \in E(\CG), \end{cases}
\end{equation}
for $x,y \in V(\CG)$ and where $d(x)$ is the degree of $x$.
 This assumption guarantees that we avoid issues of periodicity.
\begin{figure}[ht]
\includegraphics[height=4cm]{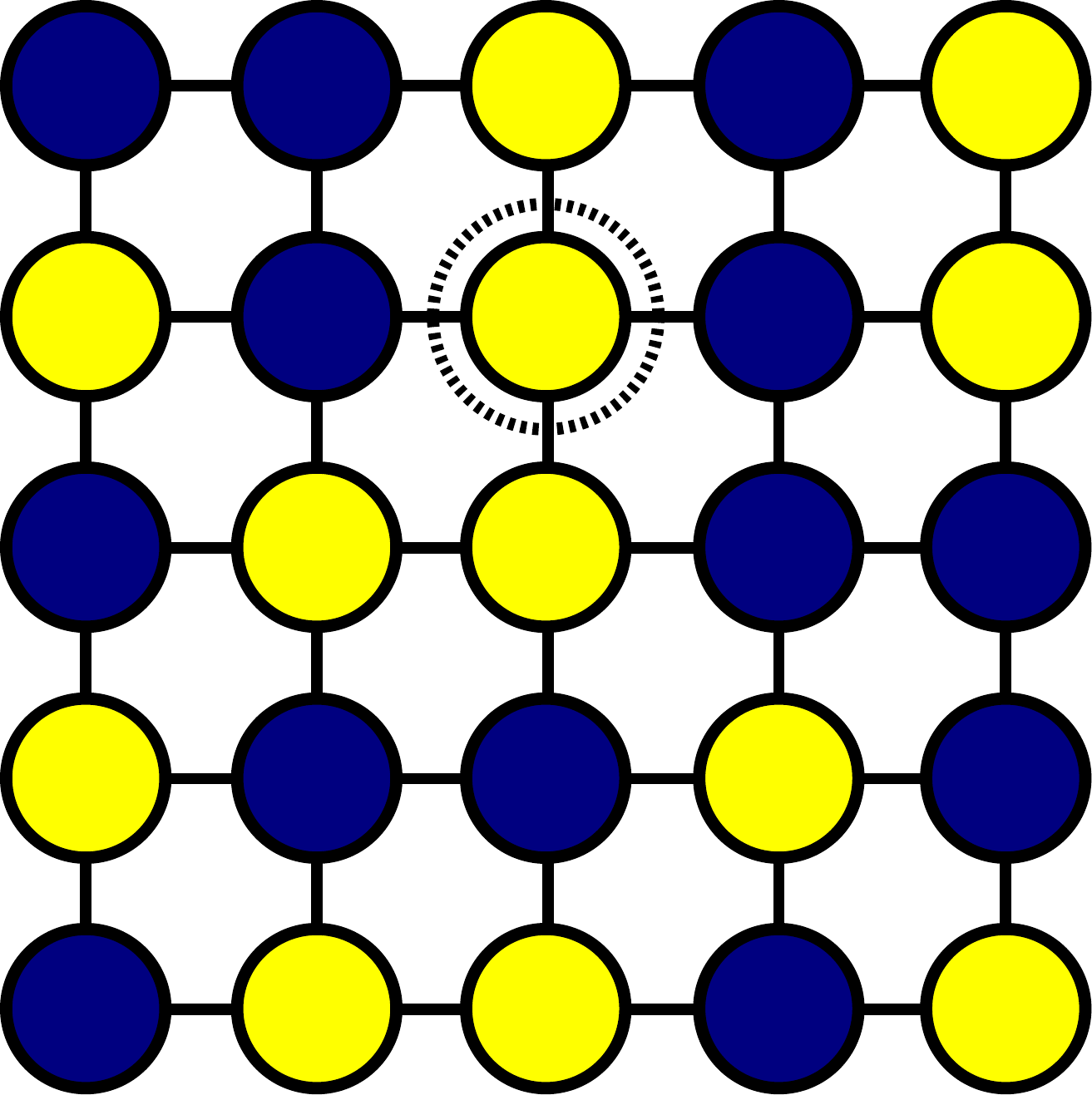}
\caption{A typical state of the lamplighter walk on the $2$-dim torus on $5$ vertices. Lamps that are `on' (resp. `off') are drawn in yellow (resp. blue) and the position of the lamplighter is marked by the dashed circle.}
\end{figure}

We will study below the \emph{total variation} ($TV$) mixing time and the \emph{relaxation time} of these walks. The relaxation time is a more algebraic point of view of mixing,  that looks at the spectral behavior of the transition matrix $P$. Namely, since $P$ is a stochastic matrix, $1$ is the main eigenvalue and all the other eigenvalues of $P$ lie in the complex unit disk. If further the chain is reversible, then the eigenvalues are real and it makes sense to define the relaxation time of the chain by
\[ \trel(G):= \frac{1}{1-\la_2}, \] where $\la_2$ is the second largest eigenvalue of the chain.

In general it is known that for a reversible Markov chain the asymptotic behavior of the relaxation time, the $TV$ and a third notion, the \emph{uniform} mixing time, which is mixing in $\ell_\infty$ norm, can significantly differ, i.e. in terms of the size of the graph $G$ they can have different asymptotics.  More precisely, we have
\[ \trel(G) \le \tmix^{TV}(G, 1/4) \le \tmix^{u}(G, 1/4), \]
see \cite{AF02} or \cite{LPW08}. The lamplighter walk described above is an example where these three quantities have different order of magnitude in terms of $|G|$.

Throughout, we use the superscript $\diamond$ to specify that a quantity belongs to the lamplighter walk, that is, the underlying graph is $\Z_2\wr G$.
In order to state our general theorems, we first need to review some basic terminology from the theory of Markov chains.
Let $P$ be the transition kernel for a lazy random walk on a finite, connected graph $G$ with stationary distribution $\pi$.

The \emph{maximal hitting time} of $P$ is
\begin{equation}
\label{eqn::thit_definition}
 \thit(G) := \max_{x,y\in V(G)} \E_x[\tau_y],
\end{equation}
where $\tau_y$ denotes the first time $t$ that $X(t) = y$ and $\E_x$ stands for the expectation under the law in which $X(0) = x$.
The random cover time $\tau_{\text{cov}}$ is the first time when all vertices have been visited by the walker $X$, and the cover time $\tcov(G)$ is
\begin{equation}
\label{eqn::tcov_definition}
 \tcov(G) := \max_{x\in V(G)} \E_x[\tau_{\text{cov}}].
\end{equation}

Then we have the following two theorems (from \cite{PR04}):
\begin{theorem}
\label{thm::main_relax}
Let us assume that $G$ is a regular, connected graph.
Then there exist universal constants $0<c_1\le C_1<\infty$ such that the relaxation time of the  lamplighter walk on
$\Z_2\wr G$ satisfies
\begin{align}
 c_1\thit(G)  \leq \trel(\Z_2\wr G) \leq C_1\thit(G), \label{eqn::trel_main}
\end{align}
\end{theorem}

\begin{theorem}\label{thm::main_mixing}
Let $G$ be a regular connected graph.
Then there exist universal constants $0<c_2\le C_2<\infty$ such that the mixing time of the lamplighter walk on $\Z_2\wr G$ satisfies
\be\label{eqn::tmix_main}
 c_2 \tcov(G) \leq \tmix(\Z_2\wr G) \leq C_2 \tcov(G).
\ee
\end{theorem}
\subsection{Proofs} Here we modify the proof that can be found in \cite{KP12} for more general lamp graphs to the setting where the lamp graph is $\Z_2$.
We start by constructing an `almost' stationary time $\tau^\diamond$ for the lamplighter walk. More specifically, the first refreshment of a lamp at site $v$ is a strong stationary time on the copy at $v$ of the two-state Markov chain on $\{0,1\}$, and we stop the chain when all lamps reach their individual stopping time, i.e. exactly when we cover all vertices. At
$\tau_{\text{cov}}$, the lamps are already stationary, but the position of the walker not necessarily.


It is easy to see that the state of the lamps are already stationary when $\tau_{\text{cov}}$ has happened, that is, for any starting state $(\un f_0,  x_0)$
 \be \label{eq::taucov} \Pv_{(\un f_0,x_0)}\left[ X^\diamond_t =(\un f, x), \tau_{\text{cov}}=t \right] =  2^{-|G|}\cdot  \Pv_{(\un f_0,x_0)}\left[ X_t=x, \tau_{\text{cov}}= t\right]. \ee
Further, if a lamp is in state $x$, then $1-x$ is a halting state for the two state Markov chain. From this it is not hard to see that the vectors $\left( (1-f_0(v))_{v\in G},y\right)$ are halting state vectors for $\tau_{\text{cov}}$ and initial state $(\un f_0, x_0)$ for every $x_0, y\in G$.
\begin{lemma}\label{lem::sep_lower_bound}
For the separation distance on the lamplighter chain $\Z_2\wr G$ the following lower bound holds:
\[  s_{(\un f_0,x_0)}^\diamond(t) \ge \Pv_{(\un f_0,x_0)}\left[ \tau_{\mathrm{cov}} > t \right].\]
\end{lemma}
\begin{proof}
Observe that reaching the halting state vector $((1-f_0(v))_{v\in G}, x)$ implies the event $\tau_{\mathrm{cov}}\le t$ so we have
\be \label{eq::cor_est} \frac{\Pv_{(\un f_0, x_0)}\left[ X^\diamond_t= ((1-f_0(v))_{v\in G}, x)\right] }{ \pi_G(x)2^{-|G|} } =
\frac{\Pv_{(\un f_0, x_0)}\left[ X^\diamond_t= ((1-f_0(v))_{v\in G}, x), \tau_{\mathrm{cov}}\le t  \right] }{ \pi_G(x) 2^{-|G|} } \ee
 Now pick a vertex $x=x_{x_0,t}\in G$ which minimizes $\Pv\left[ X_t=x_{x_0,t}| \tau_{\mathrm{cov}}\le t\right]/ \pi_G(x_{x_0,t})$. This quotient is less than $1$ since both the numerator and the denominator are probability distributions on $G$. Then, using this and \eqref{eq::taucov}, $1$ minus the right hand side of \eqref{eq::cor_est} equals
\[ 1- \frac{\Pv_{(\un f_0,x_0)}\left[ X_t = x_{x_0,t} |  \tau^\diamond \le t\right] \Pv_{(\un f_0,x_0)}[\tau^\diamond\le t] }{ \pi_G(x_{x_0,t})} \ge 1-\Pv_{(\un f_0,x_0)}\left[ \tau^\diamond \le t \right].  \]
The separation distance is larger than the left hand side of \eqref{eq::cor_est} by definition, and the proof of the claim follows.
\end{proof}

With this lemma in hand, we can already prove the lower bound in Theorem \ref{thm::main_mixing}.

\begin{proof}[Proof of the lower bound for mixing time of $\Z_2\wr G$]
Let us set $t:=6 \tmix(\Z_2\wr G)$.
Then Lemma \ref{lem::sep_lower_bound} and Lemma \ref{lem::sep_tv} yields us the following sequence of inequalities:
\[  \Pv_{(\un f_0,x_0)}\!\left[ \tau_{\text{cov}}\! >\! 6 \tmix(\Z_2\wr G) \right]\!\le s_{(\un f_0,x_0)}^\diamond\!\left(6 \tmix(\Z_2\wr G)\right)\! \le\!  4 d^\diamond(3 \tmix(\Z_2\wr G)) \le \frac{1}{2}, \]
where in the last inequality we used the sub-multiplicativity property \eqref{eq::dsubmulti}. Note that this estimate is independent of the starting state. Comparing the left and right hand sides, we conclude that we can run the chain in blocks of $6 \tmix(\Z_2\wr G)$, and in each block the graph $G$ is covered with probability at least $1/2$. Thus, $\tau_{\text{cov}}$ can be stochastically dominated by $6 \tmix(\Z_2\wr G) \mathrm{Geo}(1/2)$. Taking expected value yields
\[ \tcov(G) \le 12 \tmix(\Z_2\wr G),\]
finishing the lower bound with $c_2=1/12$.
\end{proof}

\begin{proof}[Proof of the upper bound for mixing time of $\Z_2\wr G$]
The proof of the upper bound in Theorem \ref{thm::main_mixing} is very similar, we just need to make the position of the lamplighter also stationary. We can achieve this by waiting an extra strong stationary time $\tau_G$ after $\tau^\diamond\equiv \tau_{\text{cov}}$ has happened. The existence of a \emph{separation optimal} strong stationary time on $G$ is ensured by Theorem \ref{thm::AD}.

More precisely, we have
\begin{lemma}\label{lem::opt_sst_lamplighter_2}
Let $\tau_G(x)$ be a separation-optimal strong stationary time for $G$ starting from $x\in G$ and define $\tau^\diamond_2$ by
\be\label{eq::tauhg2} \tau^\diamond:= \tau_{\mathrm{cov}}+\tau_G(X_{\tau_{\mathrm{cov}}}), \ee
where the chain is re-started at $\tau_{\mathrm{cov}}$ from $(\un F_{\tau^\diamond},X_{\tau^\diamond})$, run independently of the past and $\tau_G$ is measured in this walk.
Then, $\tau^\diamond$ is a strong stationary time for $\Z_2 \wr G$.
\end{lemma}
The proof of this lemma is omitted here since it is not difficult but quite long, see \cite{KP12}.

With this lemma in hand, we can apply \eqref{eq::sep_ineq} -- the relation between separation distance and strong stationary times --  to get
\be\label{eq::dtupper} d_{(\un f_0,x_0)}^\diamond(t)\le s_{(\un f_0,x_0)}^\diamond(t) \le \Pv_{(\un f_0,x_0)}\left[ \tau_{\text{cov}}+\tau_G(X_{\tau_{\text{cov}}})>t\right].\ee
Now set $t=8\tcov(G)+10 \tmix(G)$.
Then by a union bound the right hand side in \eqref{eq::dtupper} is at most
 \be\label{eq::stupper} \Pv_{x_0}\left[ \tau_{\text{cov}}> 8 \tcov(G )\right] + \max_{v\in G}\Pv_v\left[ \tau_G> 10 \tmix(G)\right]. \ee
 The first term on the right hand side is at most $1/8$ by Markov's inequality, and for the second term, since $\tau_G$ is separation-optimal, (i.e. it is equality in \eqref{eq::sep_ineq}), we can put
 \[ \Pv_v\left[ \tau_G> 10 \tmix(G)\right]= s_G(10 \tmix(G)) \buildrel {\star}\over{\le} 4 d_G(5 \tmix(G)) \buildrel {\triangle}\over{\le} 4 \left(\frac{2}{4}\right)^5 = \frac18, \]
 uniformly over the starting state $v$. In the inequality with $\star$ we used Lemma \ref{lem::sep_tv}, and  the one with $\triangle$ we used the sub-multiplicativity \eqref{eq::dsubmulti}.
 Combining this estimate with \eqref{eq::stupper} and \eqref{eq::dtupper} and the fact that $\tmix(G)\le  \thit(G)\le \tcov(G)$ for all reversible chains (see \cite[Chapter 10.5,11.2]{LPW08}),  yields that
 \[ \tmix(\Z_2\wr G) \le 8\tcov(G)+10 \tmix(G)\le 18 \tcov(G).\]This finishes the proof of the upper bound with $C_2=18$. \end{proof}
Now we turn to investigate the relaxation time of $\Z_2\wr G$. To do so, we will use Lemma \ref{lem::dtlimlambda} and investigate the behavior of $s(t)^{1/t}$ as $t \to \infty.$
\begin{proof}[Proof of the upper bound for relaxation time of $\Z_2\wr G$]
To prove the upper bound, we will estimate the tail behavior of the strong stationary time $\tau^\diamond=\tau_{\text{cov}}(G)+\tau_G(X_{\tau_{\text{cov}}})$ in Lemma \ref{lem::opt_sst_lamplighter_2}, relate it to $s^\diamond(t)$, the separation distance on $\Z_2\wr G$. We will use $\Pv$ for $\Pv_{(\un f, x)}$ for notational convenience.
Combining \eqref{eq::sep_ineq} by union bound we have
\begin{align} s_{(\un f, x)}^{\diamond}(t)&\le \Pv_{(\un f,x)}\left[\tau^\diamond>t\right]  \label{eq::stestimate} \\
&\le \ \Pv[\tau_{\text{cov}}(G)> t/2]\label{eq::cov_not}\\
&+ \max_{y\in G}\Pv_{y}\left[\tau_G>t/2\right] \label{eq::fourth_term}
\end{align}
We write $\tau_w$ for the hitting time of $w\in G$. We claim that the first term \eqref{eq::cov_not} can be bounded from above by:
\be\label{eq::not_cov_estimate} \Pv[\tau_{\text{cov}}(G)>t/2] \le \Pv[\exists w: \tau_w>t/2] \le |\CG| 2 e^{-\tfrac{\log 2}{4}\tfrac{t}{\thit(G)} },\ee
where $\thit(G)$ is the maximal hitting time of the graph $\CG$, see \eqref{eqn::thit_definition}.
To see this, use Markov's inequality on $\tau_w$ to obtain that for all starting states $v\in \CG$ we have
$\Pv_v[ \tau_w> 2\thit(G) ] \le 1/2 $,  and then run the chain on $G$ in blocks of $2\thit(G)$. In each block we hit $w$ with probability at least $1/2$, so we have
\[ \Pv_v[ \tau_w> K (2\thit(G)) ] \le \frac{1}{2^K}.\]
To get a similar bound for arbitrary $t$,  we can move from $\fl{ t/2\thit(G)}$ to $t/2\thit(G)$ by adding an extra factor of $2$, and \eqref{eq::not_cov_estimate} immediately follows by a union bound.

For the second term \eqref{eq::fourth_term} we prove the following upper bound:
\be\label{eq::not_tauG_estimate}  \Pv_v\left[ \tau_{G}\ge t/2 \right] \le |G| e^{-\tfrac{t}{2\trel(G)}}.\ee

First note that according to Lemma \ref{lem::dtlimlambda}, the tail of the strong stationary time $\tau_{G}$ is driven by $\la_{G}^t$ with $\la_G$ being the second largest eigenvalue of the lazy random walk on $G$. More precisely, using the first line in \eqref{eq::la_lim_upper} we have that for any initial state $v\in G$:
\[  \Pv_v\left[ \tau_G \ge t/2 \right] \le s_{G}\left(t/2\right) \le  \frac{1}{\pi_{\min}(G)} \la_{G}^{t/2}
\le |G| \exp\left\{-\frac{(1-\la_{G}) t}{2}\right\},  \]
where we used that regularity of $G$ implies $\pi_{\min}(G)=|G|^{-1}$, and the inequality $1-x\le e^{-x}$ for $x=1-\la_G$. Next we combine the bounds in \eqref{eq::not_cov_estimate} and \eqref{eq::not_tauG_estimate} on \eqref{eq::stestimate} with the second inequality in \eqref{eq::la_lim_upper} to estimate the second largest eigenvalue on $\Z_2\wr G$ as follows:
 \be\label{eq::la_estimate}  |\la_2|^t\le 2d^\diamond(t) \le 2 s^\diamond(t) \le  4 |\CG| \exp\left\{-\frac{\log2}{4}\frac{t}{\thit(G)} \right\} +2|G| \exp\left \{- \frac{t}{2\trel(G)} \right\}.
\ee
In the final step we apply Lemma \ref{lem::dtlimlambda}: we  take the power $1/t$  and limit as $t$ tends to infinity with fixed graph size $|\CG|$ on the right hand side of \eqref{eq::la_estimate} to get an upper bound on $\la_2$.  Then we use that $(1-e^{-x})\le x +o(x)$ for small $x$ and obtain the bound on $\trel(\Z_2\wr G)$ finally:
\[ \trel(\Z_2\wr G) \le \max \left \{ \frac{4}{\log 2} \thit(G), 2 \trel(G) \right\}.\]
 Then, taking into account that  $\trel(G)\le c \tmix(G)\le C \thit(G)$ holds for any lazy reversible chain (see e.g. \cite[Chapter 11.5,12.2]{LPW08}), we can ignore the second term.
\end{proof}

\begin{proof}[Proof of the lower bound for the relaxation time]
We do not include the proof of the lower bound of the relaxation time in these lecture notes since it is based on a somewhat different technique:  it relies on the analysis of the Dirichlet form of the lamplighter walk, with an appropriately chosen test-function $f$. For more details see \cite[Chapter 19.2]{LPW08} for $0-1$ lamps or \cite{KP12} for general lamp graphs.
\end{proof}

\subsection{Generalized lamplighter walks}
One can think of a generalisation of lamplighter walks of the following form: instead of $0-1$ lamps, put at each site of the base graph $G$ an identical copy of machine, whose states are represented by a \emph{lamp graph} $H$ with a fixed Markov chain transition matrix $Q$ on $H$. The walker then does the following: as he follows a simple random walk on the base graph, he modifies the state of the machines along his path randomly according to the transition matrix $Q$. The state space in this case is a vector of the states of each machine plus the position of the walker. We denote the corresponding graph by $H\wr G$. One step of the lamplighter walk is then: refresh the machine of the departure site, move one step on the base graph, refresh the machine on the arrival site. With this dynamics, one can show that the product measure of the stationary measure of $Q$ over $v\in G$ multiplied by $\pi_G$ is stationary for this dynamics and the chain is reversible. We denote the resulting graph by $H\wr G$.
We can characterise the relaxation time of such walks as follows, from \cite{KP12}:
\begin{theorem}
\label{thm::main_relax2}
Let us assume that $\CG$ and $\Hh$ are connected graphs with $\CG$ regular and the Markov chain on $\Hh$ is lazy, ergodic and reversible.
Then there exist universal constants $0<c_1,C_1<\infty$ such that the relaxation time of the  generalized lamplighter walk on
$\Hh\wr \CG$ satisfies
\begin{align}
 c_1  \leq \frac{\trel(\Hh\wr \CG)}{  \thit(\CG) +|\CG|\trel(\Hh)}  \leq C_1, \label{eqn::trel_main2}
\end{align}
\end{theorem}
\begin{theorem}\label{thm::main_mixing2}
Assume that the conditions of Theorem \ref{thm::main_relax2} hold.
Then there exist universal constants $0<c_2,C_2<\infty$ such that the mixing time of the  generalized lamplighter walk on $\Hh\wr \CG$ satisfies
\be\label{eqn::tmix_main2}\begin{aligned}
 &c_2 \big( \tcov(G)+\trel(H) |G| \log |G| + |G|\tmix(H)\big) \leq \tmix(\Hh\wr \CG), \\
 &\tmix(H\wr G) \leq C_2 \left( \tcov(\CG) +|\CG|\tmix(\Hh,\frac{1}{|\CG|})\right).
\end{aligned}
\ee
If further the Markov chain is such that
\begin{description}
\item[(A)] \label{ass::sst_exist} There is a strong stationary time $\tau_H$ for the Markov chain on $H$ which possesses a halting state $h(x)$ for every initial starting point $x\in H$,
\end{description}
then the upper bound of \eqref{eqn::tmix_main2} is sharp.
\end{theorem}
The proofs above for $0-1$ lamps can be modified to work for general lampgraphs $H$. In this case, we also have to construct an `almost' stationary time similar to $\tau_{\mathrm{cov}}$ and a true stationary time $\tau^\diamond$.  The first can be done by using copies of a separation-optimal $\tau_H(v)$, $v\in G$, such that each $\tau_H(v)$ is measured only using the transition steps of the chain on the machine $H_v$ at $v\in G$. Then we wait until all of the $\tau_H(v)$-s have happened. One can then show that this time is `almost' stationary in the sense that reaching it, the state of the lamp-graphs are stationary, but the position of the walker is not. A similar estimate to that in Lemma \ref{lem::sep_lower_bound} gives a lower bound on the separation distance.  Adding an extra $\tau_G$ again gives a `true' strong stationary time $\tau^\diamond$.

In most estimates for the mixing and relaxation time of $H\wr G$ we can use these two stopping times, but there are new terms arising: one has to estimate the local-time structure of the base graph and also the behaviour of $\tau_H$-s. The proofs are worked out in \cite{KP12}.

We mention that the upper and lower bound on the mixing time for $H\wr G$ do match for a wide selection of $H$ and $G$, but not in general. It remains an open problem to give a general formula for the mixing time.

\section{Varopoulos-Carne long range estimate}
In this section we move on to give a general bound on transition probabilities of SRW on graphs. Later, we will use this estimate to determine the speed of RW on different groups.
Let $P = (p(x,y))$ be a transition probability matrix on state space $\CS$. Assume reversibility, i.e., that $\pi(x) > 0$ and $\pi(x) p(x,y) = \pi(y) p(y,x)$ for all $x,y \in \CS$.
\\

We may consider  $\CS$ as the vertex set of an undirected graph where $x,y$ are adjacent iff $p(x,y)>0$. Let $\rho(x,y)$ denote the graph distance in $\CS$. We assume $\CS$ is locally finite (each vertex has finite degree).
We now state the \emph{Varopoulos-Carne long-range estimate}:
\begin{theorem}[Varopoulos-Carne]\label{thm::varcarne}
$\forall$ $x,y \in \CS$ and $\forall t\in \N$,
\begin{equation} \label{vc}
p^t(x,y) \le 2\sqrt{\frac{\pi(y)}{\pi(x)}} \cdot\mathbb{P}(S_t \ge \rho(x,y)) \le 2 \sqrt{\frac{\pi(y)}{\pi(x)}} e^{\frac{-\rho^2(x,y)}{2t}} \,,
\end{equation}
where $(S_t)$ is simple random walk on $\Z$.
\\
\end{theorem}
\begin{remark}\normalfont The Varopoulos-Carne estimate gives good bounds on transition probabilities between vertices that are far away from each other. Another, short-distance estimate is the following, that can be found in various forms in the literature, see  e.g.\ \cite[Theorem 17.17]{LPW08}.
Let $P$ be the transition matrix of lazy random walk on a graph of maximal degree $\Delta$. Then
\[\big|P^t(x,x)-\pi(x)\big|\le \frac{\sqrt{2} \Delta^{5/2}}{\sqrt{t}}. \]
\end{remark}
\begin{proof}[Proof of Theorem \ref{thm::varcarne}]
We start by reducing to the finite case.
Fix $t$ and $x$. Denote $\hat{\CS} = \{z: \rho(x,z) \le t\}$.
\\
Now $\forall z,w \in \hat{\CS}$, consider the modified transition matrix
\begin{displaymath}
   \hat{p}(z,w) =  \left\{
     \begin{array}{lr}
      p(z,w) & : z \ne w\\
       p(z,z) + p(z,\CS - \hat{\CS}) & : z = w
     \end{array}
   \right.
\end{displaymath}
Then $\hat{p}$ is reversible on $\hat{\CS}$ with respect to $\pi$.
Since in $t$ steps, the walk started at $x$ cannot exit $\hat{\CS}$, it suffices to prove the inequality for $\hat{\CS}$ in place of $\CS$, so we may assume that $\CS$ is finite.

Let $\xi = \cos{\theta} = \frac{e^{i\theta}+e^{-i\theta}}{2}$. Taking the $t$-th power, we see that the coefficients of the binomial expansion are exactly the transition probabilities of SRW on $\Z$, which gives
$$\xi^t = \sum_{k=-t}^{t} \mathbb{P}(S_t = k) e^{ik\theta}.$$
By taking the real part, we get
\begin{equation} \label{bin}
\xi^t= \sum_{k=-t}^{t} \mathbb{P}(S_t = k) \cos{k\theta}.
\end{equation}
Now denote $Q_k(\xi) = \cos{k\theta}$. Observe that $Q_0(\xi) = 1, Q_1(\xi) = \xi$, and the identity   $$\cos{(k+1)\theta}+  \cos{(k-1)\theta}=2\cos{\theta}  \cos{k\theta}$$ yields that   $Q_{k+1}(\xi) + Q_{k-1}(\xi) = 2 \xi Q_k(\xi)$ for all $k \ge 1$. Thus induction gives that $Q_k$ is polynomial of degree $k$ for all $k \ge 1$; these are the celebrated Chebyshev polynomials. Further, since $Q_k(\xi)=\cos(k \theta)$ for $\xi=\cos\theta \in[-1,1]$ implies the fact that $|Q_k(\xi)| \le 1$ for $\xi\in[-1,1]$.
\\
Using the symmetry of cosine function, we can rewrite (\ref{bin}) in the form
$$\xi^t = \sum_{k=-t}^{t} \mathbb{P}(S_t = k) Q_{|k|}(\xi),$$
which is an identity between polynomials. Applying it to the transition probability matrix $P$ on $\CS$, we infer that
\begin{equation}
\label{mat}
P^t = \sum_{k=-t}^{t} \mathbb{P}(S_t = k) Q_{|k|}(P)
\end{equation}

We know that all eigenvalues of $P$ are in $[-1,1]$. Furthermore, the eigenvalues of $Q_k(P)$ have the form $Q_k(\lambda)$, where $\lambda$ is an eigenvalue of $P$, so they are also in $[-1,1]$. Hence $\|Q_k(P) v\|_\pi \le \|v\|_{\pi}$ for any vector $v$, where $\|v\|_\pi^2=\sum_{x\in\CS} v(x)^2 \pi(x).$
\\
Using this contraction property we can write
\[Q_k(P)(x,y) = \frac{\langle \delta_x, Q_k(P) \delta_y \rangle_{\pi}}{\pi(x)} \le \frac{\|\delta_x\|_\pi \|\delta_y\|_\pi}{\pi(x)} \le \frac{\sqrt{\pi(x)}\sqrt{\pi(y)}}{\pi(x)} = \sqrt{\frac{\pi(y)}{\pi(x)}}.\]
Note that $P^k(x,y)\!=\!0\  \forall k<\rho(x,y)$ implies $Q_k(P)(x,y) = 0$ for $k < \rho(x,y)$.

Hence, by (\ref{mat}), we have
\[p^t(x,y) = \sum_{|k| \ge \rho(x,y)} \mathbb{P}(S_t = k) Q_{|k|}(P)(x,y) \le  \sum_{|k|\ge \rho(x,y)} \mathbb{P}(S_t = k) \sqrt{\frac{\pi(y)}{\pi(x)}} ,\]
proving the first inequality in (\ref{vc}).
The second inequality in (\ref{vc}) is an application of the well-known Bernstein-Chernoff bound
\begin{equation} \label{bernstein}
\Pv(S_t \ge R) \le e^{-R^2/(2t)}.
\end{equation}
For the reader's convenience we recall the proof. Suppose that $\Pv(X=1)=1/2=\Pv(X=-1)$. Then
\[
\Ev(e^{\lambda X})  = \frac{ e^{\lambda} +  e^{-\lambda}}{2} = \sum_{k=0}^{\infty} \frac{\lambda^{2k}}{(2k)!}\\ \le \sum_{k=0}^{\infty} \frac{\lambda^{2k}}{2^k k!} = e^{\lambda^2/2}.
\]
Therefore,
\[\Ev(e^{\lambda S_{t}})= (\Ev(e^{\lambda X}))^t \le e^{t\lambda^2/2}.\]
Finally, by Markov's inequality,,
\[\Pv(S_t \ge R) = \Pv(e^{\lambda S_t} \ge e^{\lambda R}) \le
e^{-\lambda R} \cdot e^{t\lambda^2/2}.\]
Optimizing, we choose $\lambda = R/t$, and (\ref{bernstein}) follows.

\end{proof}
\section{Speed of RW on groups and harmonic functions}
In this section we characterize the speed of random walk on groups in terms of bounded harmonic functions. For more on this topic see Chapter 13 in  \cite{LPbook}.

Let $G$ be a (finite or countable) group, with finite generating set $S$. We assume $S = S^{-1}$, and $d = |S|$. Recall the right-Cayley graph on $G$ is given by $x\sim y \Leftrightarrow y\in xS$, and the corresponding simple random walk (SRW) has
\be\label{def::srw}p_{\sss{SRW}}(x, y) =
\begin{cases}
\frac{1}{d}, & \mbox{ for } y\in xS,\\
0, & \mbox{ otherwise.}
\end{cases}\ee
We define the lazy random walk (LRW) to avoid periodicity issues:
\be\label{def::lazy}p(x,y)=
\begin{cases}
\frac12, & \mbox{ for } y=x \\
\frac{1}{2d}, & \mbox{ for } y\in xS,\\
0, & \mbox{ otherwise.}
\end{cases}\ee
That is, the transition matrix $P= (P_{\sss{SRW}} +I)/2$.
We call $e\in G$ the origin, and denote $\ldist$ the graph distance in $G$.  We write simply $\ldist(e, x)=|x|$.
\begin{definition}
The speed of random walk on $G$ is defined as
\[  v(G):=\lim_{n\to \infty} \frac{\Ev|X_n|}{n}= \text{a.s. } \lim_{n\to \infty}\frac{ |X_n|}{n}.\]
\end{definition}
This definition is valid, since the distance is subadditive by the triangle inequality and the transitivity of $G$:
\[  \ldist(e, X_{n+m}) \le \ldist(e, X_{n}) + \ldist(X_n, X_{n+m}) \ {\buildrel {d}\over =} \  \ldist(e, X_{n}) + \ldist(e, X_{m}).\]
Taking expectation yields that the expected distance is submultiplicative, hence the speed exist.

The main goal of here is to characterize \emph{when is the speed positive?} But first some examples:
\begin{example} For every $d$,  $v(\Z^d)=0$. This is easy to see since $(\Ev|X_n|)^2\le \Ev(|X_n|^2)= \sum_{i=1}^n \Ev(|Y_i|^2)= n$ by denoting $Y_i$ the independent unit length increment of the walk at step $i$.
\end{example}
\begin{example} The speed on the infinite $d$-ary tree $\Pi_d$ is $v(\Pi_d)=\frac{d-2}{d}$. In each step of the walk, there are $d-1$ edges increasing the distance from the root by $+1$ and exactly $1$ edge decreasing the distance, hence the speed is ${d-2}{d}$ for non-lazy RW and $\frac{d-2}{2d}$ for lazy RW.
\end{example}
The third example needs some definitions:
\begin{definition}A state of the lamplighter group $G_d$ on $\Z^d$ is defined as $(S,x)$ where $S \subset\Z^d$ is a finite subset  of vertices and $x\in \Z^d$ is the position of a \emph{marker} or lamplighter.  Every state in $G_d$ is connected to $2d+1$ other states in $G_d$: either the marker moves to a uniformly chosen neighbour of $x$ or it switches the lamp at $x$: i.e.\ removes $x$ from $S$ if $x\in S$, and adds $x$ to $S$ if $x\notin S$. The origin in this walk is $(\emptyset, \un 0)$, i.e. all lamps off, marker at the origin.
\end{definition}
The set $S$ describes which `lamps' are on, and the marker can switch lamps only along his path. He either moves on the base graph $\Z^d$ or switches the lamp where he currently is.
\begin{example} The speed of the lamplighter walk on $G_1$ and $G_2$ is zero, while $v(G_d)>0$ for $d\ge3$.
\end{example}
\begin{proof} For $G_1$ we can use the marginal distribution of the marker is just a SRW on $\Z$, hence its range up to time $n$ is whp less than $c \sqrt{n} \log n$. Thus, any state that the lamplighter can reach in $n$ steps has at most only a connected set of on-lamps of size $c \sqrt{n} \log n$. This has distance at most $K \sqrt{n} \log n$ from the origin, since the marker can just walk along its range, switch off each lamp that is on and return to the origin, taking at most $K \sqrt{n} \log n$ steps for some $K>0$.

For $G_2$, the range of SRW on $\Z^2$ is whp $n/\log n$, so the same argument can be applied to show that the speed is zero.

For $d>3$, the range of SRW on $\Z^d$ is linear in $n$, and with positive probability there are going to be a linear number of lamps on, hence the speed is positive, too.
\end{proof}
\subsubsection*{Discussion}
We see that it is not the growth rate that characterizes the speed: trees and lamplighter groups both grow exponentially. What does characterize the speed?
the answer is given by \emph{bounded harmonic functions}.

\subsection{Bounded harmonic functions and tail $\sigma$-algebras}
We start with a definition:
\begin{definition}
We say that a bounded function $u: G \to \R$ is harmonic for the simple random walk on $G$ if
\[ u(x) = \frac{1}{d} \sum_{y\sim x} u(y),\]
that is we have  $u= P_{\sss{SRW}}u=Pu$.
\end{definition}
We define the tail $\sigma$-algebra as $\mathcal T= \bigcap_n \sigma(X_n, X_{n+1}, \dots)$. $\CT$ contains all events which are independent of the trajectory up to any fixed finite time.
Tail events can easily generate harmonic functions, we list some examples:
\begin{enumerate}
\item On $\Pi_d$, does the RW end up eventually in a given sub-branch of the tree?
\[ u_1(x):=\Pv_x( \text{RW ends in a given sub-branch of the tree}) \]
\item On $G_d$ with $d\ge3$, is the lamp at $x$ eventually on?
\[ \ba u_2(x)&:=\Pv_x( \text{the lamp at $y$ is going to be eventually on} ) \\
u_3(x)&:= \Pv( \text{the lamps in the subset $A$ are all going to be eventually on})
\ea\]
\end{enumerate}
One can easily argue that $u_1, u_2, u_3$ are non-constant by moving the starting point $x$ further and further away from the points / sets under consideration and using transience properties of the marker.
\begin{definition}
We call $f=f(X_0, X_1, X_2, \dots)$ a \emph{tail-function} if changing finitely many values in the trajectory $(X_0, X_1, X_2, \dots)$ does not change the value of $f$.
\end{definition}
\begin{claim}
Every tail function generates a bounded harmonic function by
\[ u_f(x)= \Ev_x(f(X_0, X_1, X_2, \dots))\]
for random walk on groups or for lazy chains.
\end{claim}
\begin{proof}
We prove it for lazy chains only.
First we start with a total variation bound on binomial random variables\footnote{We set ${n \choose -1} := 0$.}:
\be\label{eq::tvbin} \ba \|\mathrm{Bin}(n,\tfrac12)\! - \!\mathrm{Bin}(n+1, \tfrac12) \|_{\TV} &\le 2^{-n-1} \sum_{k=0}^{\lfloor n/2\rfloor} \left[ 2{n \choose k}-{n+1 \choose k} \right]\\
&=2^{-n-1} \sum_{k=0}^{\lfloor n/2\rfloor} \left[ {n \choose k}-{n \choose k-1} \right] \\
&= 2^{-n-1} {n \choose \lfloor n/2\rfloor} =\frac{1+o(1)}{ \sqrt{2\pi n} }.
\ea\ee
First fix some $\ve>0$ and pick $n$ large enough such that $\frac{1+o(1)}{ \sqrt{2\pi n} } \| f\|_\infty <\ve$.
Look at two copies of the lazy walk: $(X_0, X_1, X_2, \dots )$ and $(\widetilde X_0, \widetilde X_1, \widetilde X_2, \dots)$.
We can then construct a coupling between these two trajectories by using  a non-lazy random walk $Y$, and set $X_n=Y_{\mathrm{Bin}(n, \frac12)}$ and $\widetilde X_{n+1}=Y_{\mathrm{Bin}(n+1, \frac12)}$.
The bound in \eqref{eq::tvbin} and the coupling characterisation of total variation distance \eqref{eq::tv4} tells us that we can couple these two trajectories such that $\Pv(X_n\neq\widetilde X_{n+1})\le \Pv(\mathrm{Bin}(n, \frac12) \neq \mathrm{Bin}(n+1, \frac12)) \le \ve$.
Hence, we can write
\[ \ba \left(Pu_f\right)(x) -u_f(x)&= \Ev_x( f( X_1, X_2 \dots) ) - \Ev_x( f(X_0, X_1, X_2 \dots) )\\
& =  \Ev_x( f( \widetilde X_1, \widetilde X_2 \dots) ) - \Ev_x( f(X_0, X_1, X_2 \dots) ) \\
& \le \|f \|_\infty \cdot \Pv(X_n \neq \widetilde X_{n+1})  \le \ve ,\ea \]
where in the last step we used that if the two trajectories are coupled by time $n$, then clearly they only differ in finitely many steps, and $f$ is a tail function, hence it takes the same value on $\{ X_n = \widetilde X_{n+1}\}$. Since $\ve$ was arbitrary, we get $Pu_f=u_f$, finishing the proof.
\end{proof}
The reverse direction is also true:
\begin{claim}
Every bounded harmonic function $u$ defines a tail function $f_u$ by
\[ f_u(X_0, X_1, X_2, \dots):= \limsup_{n\to \infty} u(X_n). \]
\end{claim}
\begin{proof} Since $u$ is bounded and harmonic, the function $u(X_n)$ is a bounded martingale. Hence, by the martingale convergence theorem we get that it converges.  Further, the definitions of the two claims are giving a correspondence between bounded harmonic functions and tail-functions since $u_{f_u}(x)= \Ev_x (f_u(X_0, X_1, \dots)) = \Ev_x( \limsup u(X_n)) = u(x)$ by the martingale stopping theorem.
\end{proof}
We call a $\sigma$-algebra $\CF$ trivial if $\forall A \in \CF, \ \Pv_x(A) \in \{0,1\}$.

We will need the  the following equivalence.

\begin{theorem}
For random walk on a group, the tail $\sigma$ algebra $\CT$ is trivial  if and only if every bounded harmonic function on $G$ is constant.
\end{theorem}
\begin{proof} Suppose first that $\CT$ is trivial. Let $u$ be a bounded harmonic function. Then $\limsup u(X_n)$ is a tail function, so it must be constant a.s.
By irreducibility, this constant $c$ does not depend on the starting point. Writing $u(x)=\Ev_x u(X_n)$ and passing to the limit using the bounded convergence theorem proves that $u(x)=c$ for all $x$.
This direction is valid for any irreducible Markov chain. The other direction is not hard to verify for lazy irreducible Markov chains:
Suppose all bounded harmonic functions are constant, and $A \in \CT$. Then it is easy to check that $u(x)=\Pv_x(A)$ is a harmonic function, so the L\'evy zero-one law implies that $\Pv_x(A)\in\{0,1\}$ for every $x$.
Without assuming Laziness, but using the group structure instead, one can also show that $u$ is harmonic. This can be proved using entropy or via Derriennic's zero-two law \cite{Derr76}, see Chapter 13 in \cite{LPbook} for details.
\end{proof}
\subsection*{Entropy} To state the next theorem, we need some basic properties of entropy, which we include here for the reader's convenience.
\begin{definition} The \emph{entropy} of a random variable $X$ with distribution $p_x$ on state space $\CS$ is defined as
\[H(X):= \sum_{x\in \CS} p_x \log (\frac{1}{ p_x}).\]
and the \emph{relative entropy} of measure $P$ with respect to another measure $Q$ on the same state space $\CS$ is defined as
\[ D(P| Q) :=\sum_x p_x \log \left( \frac{p_x}{q_x}\right).\]
\end{definition}
The relative entropy is always nonnegative since $\log t\le t-1$ for $t>0$, hence
\[ -D(P|Q) = \sum_{x\in \CS} p_x \log \left( \frac{q_x}{p_x}\right) \le \sum_{x\in \CS} p_x \left( \frac{q_x}{p_x} -1\right) =0. \]
Finally, the \emph{conditional entropy} is defined as the entropy of the conditional measure $p(x|y)=\frac{p_{xy}}{p_y}$, i.e. \
\[ H(X|Y) := \sum_{x,y} p(x|y) \log \left(\frac{1}{\log p(x|y)}\right).\]
We write $H(X,Y)$ for the entropy of the joint distribution of $(X,Y)$. Then it is not hard to see that
\[ H(X|Y) = H(X,Y) - H(Y) \le H(X),\]
since $H(X) + H(Y) - H(X,Y)= D(p_x \cdot p_y| p_{x,y}) \ge 0$ with equality if  and only if $X$ and $Y$ are independent.
As a corollary we get that for any three random variables $X, Y, Z$
\be\label{eq::edecrease}
H(X| Y,Z=z) \le H(X| Z=z) \Longrightarrow H(X| Y, Z) \le H(X | Z).
 \ee
It can also be shown that the uniform distribution on set $\CS$ (with $|\CS|=n$) maximizes the entropy:
\[ 0 \le D(P_x, U[\CS]) = \sum_{x\in \CS} p_x \log(n p_x) = \log n -H(X). \]

\subsection{The Kaimanovich - Vershik - Varopoulos theorem}
The next theorem is by Kaimanovich - Vershik ('83) \cite{KV83} and Varopoulos ('85) \cite{V85}.
\begin{theorem}For random walk on a group $G$, the followings are equivalent:
\begin{enumerate}
\item\label{t:1} the speed $v(G)>0$,
\item\label{t:2} $\exists$ a bounded non-constant harmonic function $u$ on $G$,
\item\label{t:3} the entropy of the walk $h= \lim_{n\to \infty} \frac{H(X_n)}{n}>0$.
\end{enumerate}
\end{theorem}
\begin{proof}
First we show \eqref{t:2}$\leftrightarrow$\eqref{t:3}.
Write  the joint entropy in two ways:
\[  \ba H(X_k , X_n) &= H(X_k) + H(X_n | X_k )=  H(X_k) + H(X_{n-k}) \\
H(X_k, X_n)&= H(X_n) + H(X_k| X_n)\ea \]
Rearranging and taking $k=1$ yields that \be\label{eq::ent}H(X_1) + H(X_{n-1})-H(X_n) = H(X_1| X_n ) = H(X_1| (X_n, X_{n+1}, \dots )),\ee
where the last equality is due to the Markov property. Since conditioning on less information increases the entropy (see \eqref{eq::edecrease}), $H(X_1| (X_n, X_{n+1}, \dots ))$ is an increasing function of $n$. So, the left hand side in \eqref{eq::ent} is also increasing, so we get that $\h_n:=H(X_n)-H(X_{n-1})$ is decreasing. Hence, $h_n \to h$ for some $h\ge 0$.
So we get, that $\frac{H(X_n)}{n}\to h$.
Now if $h>0$, then taking $n\to \infty $ in \eqref{eq::ent} gives
$H(X_1| \CT)= H(X_1)-h$, that is, conditioning on $\CT$ influences the entropy: hence $\CT$ can not be trivial.
On the other hand if $h=0$ then
$H(X_k| \CT)= H(X_k)$ for all $k$, hence, the tail $\CT$ is \emph{independent} of $\{X_1, X_2, \dots X_k\}$. Thus, it must be trivial itself.

Next we show \eqref{t:3}$\leftrightarrow$\eqref{t:1}.
Apply the Varopoulous-Carne estimate on transitive groups to see that $p_n(x)\le 2 e^{-\frac{|x|^2}{2n}}$, and use this estimate on $-\log p_n(x)$ in the definition of $H(X_n)$ to get
\[ H(X_n) = \sum_{x} p_n(x) (-\log(p_n(x)) \ge \sum_x p_n(x) (-\log2 +\tfrac{|x|^2}{2n} )\]
Rearranging terms and dividing by $n$ yields
\[ \frac{\log 2+ H(X_n)}{n} \ge \frac{\Ev|X_n|^2}{2n^2}\ge\frac12 \frac{(\Ev|X_n|)^2}{n^2}= \frac12 v(G)^2, \]
where we used Jensen's inequality in the last step. Now clearly $v(G)>0$ implies  $\frac{H(X_n)}{n}>0$.

On the other hand, we can define the spheres $S_k:= \{y \in G: |y|=k\}$ and the measure $Q(x)=\frac{2^{-k-1}}{|S_k|}$ if $x\in S_k$ is a probability measure on $G$.
We calculate the relative entropy
\[ 0\le H(Q| P^n) = \sum_x p_n(x) \log \frac{p_n(x)}{Q(x)} \le \left(\sum_{x} p_n(x) (|x| +1) \log 2d\right) - H(X_n), \]
where we used the bound $-\log Q(x)\le \log (2d)^{k+1}$ since the degree is $d$.
Now dividing by $n$ yields
\[ 0\le\frac{ (\Ev[|X_n| +1])\log 2d}{n} - \frac{H(X_n)}{n}, \]
and passing to the limit shows that if $h= \lim_n \frac{H(X_n)}{n}>0$ then the speed is also positive.
This finishes the proof.
\end{proof}

\section{Geometric bounds on mixing times}

Let $G$ be a (finite or countable) group, with finite generating set $S$. We assume $S = S^{-1}$, and $d = |S|$. Recall the right-Cayley graph on $G$ is given by $x\sim y \Leftrightarrow y\in xS$, and consider simple random walk on $G$ as in \eqref{def::srw}
Let $\ldist$ denote graph distance in $G$.

\smallskip

\begin{theorem}\label{thm::geomix}
 For simple random walk on $G = \langle S
\rangle$,

(a) If $|G| < \infty$, then $\Ev[\ldist(X_0, X_n)^2] \geq \frac{n}{2d}$ for $n \leq \frac{1}{1-\lambda}$, where $\lambda = \lambda_2$ is the second eigenvalue.

(b) If $|G| = \infty$ and $G$ is amenable, then $\Ev [\ldist(X_0, X_n)^2] \geq \frac{n}{d} $ for all $n\geq 1$.
\end{theorem}
\smallskip

\begin{remark} \normalfont  \begin{enumerate} \item The theorem is proved in Lee-Peres \cite{LP09} in the more general setting of random walks on transitive graphs. \item Part (b) for Cayley graphs was first discovered by Anna Ershler (unpublished) who relied on a harmonic embedding theorem of Mok. \item If $G$ is nonamenable, then we know that $\E\ldist(X_0, X_n) \geq cn$, so that $\E [\ldist(X_0, X_n)^2] \geq c^2 n^2$ for some constant $c>0$. \end{enumerate}
\end{remark}
Theorem \ref{thm::geomix} for finite, transitive graphs gives a very general upper bound on relaxation and mixing times of finite groups:
\begin{corollary}\label{cor::trel}
Write $\diam(G)$ for the diameter of $G=\langle S \rangle$. Then
\be\ba\label{eq::trel_lower}  \trel(G) &\le 2 d \cdot \diam(G)^2,\\
\tmix(G)&\le 2d \cdot \diam(G)^2\cdot\log |G|. \ea\ee
 \end{corollary}
 It is an open problem whether $\tmix(G)\le C d\cdot  \diam(G)^2$ holds for every transitive finite chain.
 \begin{proof}[Proof of Corollary \ref{cor::trel}] Apply part (a) of Theorem \ref{thm::geomix} with $n=\trel(G)$:
 \[ \diam(G)^2 \ge \Ev[\ldist(X_0, X_n)^2]  \ge \frac{\trel(G)}{2d}.\]
For the second inequality, use \cite[Theorem 12.3]{LPW08} stating that $\tmix(G) \le -\log(\pi_{\min}) \trel(G)$.
\end{proof}
To prove Theorem \ref{thm::geomix}, we use the following key lemma from \cite{LP09}
(that is valid for transitive graphs as well). We define the
\emph{Dirichlet forms} $Q_n(f) := \langle(I - P^n)f, f\rangle$.
\begin{lemma}\label{lem::dirichlet} For the simple random walk on $G$ as in Theorem \ref{thm::geomix} and any $f\in \ell^2(G)$, we have
$$\Ev[\ldist(X_0, X_n)^2] \geq \frac{1}{d} \frac{Q_n(f)}{Q_1(f)}\,.$$
\end{lemma}
\begin{proof}[Proof of Theorem \ref{thm::geomix} (finite case) from Lemma \ref{lem::dirichlet}] In the finite case, take $f$ as an eigenfunction such that $Pf = \lambda f$ with $\|f\|_2 = 1$. Then $Q_n(f) = 1- \lambda^n$. Using the condition $n<\tfrac{1}{1-\la}$ we can write
\begin{align*}
d\cdot \Ev[\ldist(X_0, X_n)^2] \geq \frac{1-\lambda^n}{1-\lambda} = \sum_{j=0}^{n-1} \lambda^{-j} \geq \sum_{j=0}^{n-1} \big(1 - \frac{1}{n}\big)^j \geq \sum_{j=0}^{n-1} \big(1 - \frac{j}{n}\big) \geq \frac{n}{2}\,.
\end{align*}
\end{proof}
The infinite case is harder and will be proved later.
\begin{proof}[Proof of Lemma \ref{lem::dirichlet}]
 Given $f\in \ell^2(G)$, construct $F: G \rightarrow \ell^2(G)$ by $F(x) := \{f(gx)\}_{g\in G}$. Compute (with $X_0 = x_0$)
\begin{align} \label{eq-key-identity}
\Ev& \|F(X_0) - F(X_1)\|^2_2 = \Ev \sum_{g\in G} \|f(gX_0) - f(gX_1)\|^2_2 = \sum_x\sum_y |f(x) - f(y)|^2 p(x, y) \nonumber\\
&= \sum_x\sum_y [(f(x))^2 + (f(y))^2 - 2f(x) f(y)] p(x, y) = 2
\langle(I-P)f, f\rangle = 2Q_1(f)\,.
\end{align}
Similarly, $\Ev\|F(X_0) - F(X_n)\|^2_2 = 2 Q_n(f)$.  Now, \eqref{eq-key-identity} implies that
\[ \frac{1}{d} \|F(x_0) - F(y)\|^2 _2\leq 2 Q_1(f)\] for any $x_0, y$ with $x_0 \sim y$. Thus, $F$ is Lipshitz with $\mathrm{Lip}(F) \leq \sqrt{2d Q_1(f)}$. Therefore,
$$2Q_n(f) = \Ev \|F(X_n) - F(X_0)\|^2_2 \leq (\mathrm{Lip}(F))^2 \Ev [\ldist(X_0, X_n)^2] \leq 2d Q_1(f) \Ev[\ldist(X_0, X_n)^2]\,.$$
Rearranging proves Lemma \ref{lem::dirichlet}.
\end{proof}
Now we turn to the proof of Theorem \ref{thm::geomix} for infinite $G$. We will need the following lemma:

\begin{lemma}\label{lem::distancebound}Given $f\in \ell^2(G)$,
\[ \Ev [\ldist(X_0, X_n)^2] \geq \frac{n}{d} - \frac{n^2}{2d} \frac{\|(I-P)f\|^2}{Q_1(f)}\,.\]
\end{lemma}
\begin{proof}We use Lemma \ref{lem::dirichlet}. We need to lower bound $Q_n(f)$, and show that it grows almost linearly. For this, we use the differences and bound second differences as follows:
$$\Delta_j = Q_{j+1}(f) - Q_j(f) = \langle P^j f - P^{j+1}f, f \rangle = \langle(I-P) P^j f, f\rangle = \langle P^j f, (I-P)f\rangle.$$
Thus,
\[\ba |\Delta_j - \Delta_{j-1}| &= |\langle P^{j-1}(I-P)f, (I-P)f\rangle| \\
& \leq \|P^{j-1} (I-P) f\|_2 \cdot \|(I-P) f\|_2 \leq \|(I-P)f\|^2_2 := \delta\,,\ea \]
by Cauchy-Schwarz. Now $\Delta_0 = Q_1(f)$ and $\Delta_j \geq \Delta_0 - j\delta$ whence
$$Q_n(f) = \sum_{j=0}^{n-1} \Delta_j \geq n\Delta_0 - \frac{n(n-1)}{2} \delta \geq n Q_1(f) - \frac{n^2 \delta}{2}\,.$$
Thus,
$$\frac{Q_n(f)}{Q_1(f)} \geq n - \frac{n^2 \|(I-P)f\|^2}{2 Q_1(f)}$$
and the lemma follows from Lemma \ref{lem::dirichlet}.
\end{proof}
Proving the theorem for $G$ infinite is harder; we first give the proof under an additional assumption.
\begin{assumption}\label{a::star} Suppose that $\sum_{j=0}^\infty (P^j \one_{\{x_0\}})(x): =  \mathrm{Green}(x_0, x)$ is in $\ell^2(G)$.
\end{assumption}

\begin{proof}[Proof of Theorem \ref{thm::geomix} (infinite case) assuming Assumption \ref{a::star}]
Note that Lemma \ref{lem::distancebound} gives the statement of theorem if we can find a sequence of functions $f_k$ for which $\frac{\|(I-P)f_k\|^2_2}{ Q_1(f_k)} \to 0$.

Let $\{A_k\}$ be a sequence of F\"olner sets, i.e., $\delta_k := \frac{|\partial_E A_k|}{|A_k|} \to 0$ as $k \to \infty$. Here $\partial_E A$ denotes the \emph{edge-boundary} of the set $A$, i.e. the edges between $A$ and $A^c$. Write $\psi_k = \one_{A_k}$ and $f_k = \sum_{j=0}^\infty P^j \psi_k$. Assumption \ref{a::star} implies that $f_k \in \ell^2(G)$. Note that $(I-P)f_k = \psi_k$ and $f_k(x) = \Ev_x [\sum_{j=0}^\infty \one_{\{X_j\in A_k\}}]$. If $\ldist(x, A_k^c) \geq r$, then $f_k(x) \geq r$, so combining these yields
\begin{align*}
Q_1(f_k)& = \langle(I-P)f_k, f_k \rangle = \sum_{x\in A_k} f_k(x) \geq r |\{x\in A_k : \ldist(x, A_k^c) \geq r\}| \\
&\geq r [|A_k| - d|\partial_E A_k|] = r |A_k| (1 - d\delta_k)\,.
\end{align*}
Letting $k\to \infty$ gives $\liminf_{k\to \infty} \frac{Q_1(f_k)}{|A_k|} \geq r$ whence $\frac{Q_1(f_k)}{|A_k|} \to \infty$ since $r$ was arbitrarily large.
By Lemma \ref{lem::distancebound},
$$\E [\ldist(X_0, X_n)^2] \geq \frac{n}{d} - \frac{n^2}{2d} \frac{|A_k|}{Q_1(f_k)}\,.$$
Letting $k\to \infty$ proves the theorem assuming Assumption \ref{a::star}.
\end{proof}
\subsection*{Removing Assumption \ref{a::star}}
For the next lemma, we recall that if $P$ is  transient or
null-recurrent, then we have the pointwise limit,
\begin{equation}\label{eq:pwlimit}
P^i f \to 0 \quad \textrm{for every } f \in \ell^2(V).
\end{equation}
\begin{lemma}\label{lem::psilem}
Suppose that $P$ satisfies \eqref{eq:pwlimit} and, for some $\theta
\in (0,\frac12)$, there exists an $f \in \ell^2(V)$ with $\|f\|_2=1$
and $\|Pf-f\|_2 \leq \theta$. Then there exists a $\varphi \in
\ell^2(V)$ such that
\begin{equation}\label{eq:psijs}
\frac{\|(I-P) \varphi\|^2_2}{\langle \varphi, (I-P) \varphi\rangle}
\leq 32 \,\theta.
\end{equation}
\end{lemma}
\begin{proof}[Proof of Theorem \ref{thm::geomix} for infinite $G$ without Assumption \ref{a::star}] The proof follows by picking $f:=\widetilde\psi_k=\one_{A_k}/ \sqrt{ |A_k|}$ for the F\"olner sets defined above. By picking $k$ large enough $\widetilde\psi_k$ satisfies the condition of Lemma \ref{lem::psilem} for arbitrarily small $\theta>0$, since in this case $\|f\|_2^2=|A_k|/ |A_k| =1$ and \[ \|P\widetilde \psi_k-\widetilde\psi_k\|_2^2= \sum_{x\in A_k} \Ev(X_1 \in A^c) \le \frac{|\partial_E A_k  |}{  |A_k|} = \delta_k \to 0.\] Combining then these with Lemma \ref{lem::distancebound} yields the proof.
\end{proof}
\begin{proof}[Proof of Lemma \ref{lem::psilem}]
Given $f \in \ell^2(V)$ and $k \in \mathbb N$, we define
$\varphi_{k} \in \ell^2(V)$ by
$$
\varphi_{k} = \sum_{i=0}^{k-1} P^i f.
$$
First, using $(I-P)\varphi_{k} = (I-P^k) f$ and the fact that $P$ is
a contraction, we have
\begin{equation}
\label{eq:lapest} \|(I-P)\varphi_{k}\|^2_2 \leq 4 \|f\|^2_2.
\end{equation}

On the other hand,
\begin{eqnarray*}
\langle \varphi_{k}, (I-P)\varphi_{k} \rangle &=& \langle \varphi_{k}, (I-P^k) f\rangle \\
&=&
\left\langle (I-P^k) \sum_{i=0}^{k-1} P^i f, f \right\rangle \\
&=& \langle 2 \varphi_{k}-\varphi_{2k}, f \rangle,
\end{eqnarray*}
where in the second line we have used the fact that $I-P^k$ is
self-adjoint. Combining this with \eqref{eq:lapest} yields
\begin{equation}\label{eq:ratio1}
\frac{\|(I-P)\varphi_{k}\|^2_2}{\langle \varphi_{k},
(I-P)\varphi_{k}\rangle} \leq \frac{4 \|f\|^2_2}{\langle 2
\varphi_{k}-\varphi_{2k},f\rangle}.
\end{equation}
The following claim will conclude the proof.

\medskip

\noindent{\bf Claim}: There exists a $k \in \mathbb N$ such that
\begin{equation}\label{eq:wantit}
\langle 2 \varphi_{k} - \varphi_{2k}, f \rangle \geq
\frac{1}{8\theta}.
\end{equation}
It remains to prove the claim. By assumption, $f$ satisfies
$\|f\|_2=1$, and $\|Pf-f\|_2 \leq \theta$.  Since $P$ is a contraction,
we have $\|P^j f - P^{j-1} f\|_2\leq \theta$ for every $j \geq 1$,
and thus by the triangle inequality, $\|P^j f - f\|_2 \leq j \theta$
for every $j \geq 1$. It follows by Cauchy-Schwarz that $\langle f,
(I-P^j) f \rangle \leq j\theta$, therefore
\begin{equation*}\label{eq:innerprod}
\langle f, P^j f \rangle \geq 1-j \theta.
\end{equation*}
Thus for every $j \geq 1$,
$$
\langle \varphi_{2^j}, f \rangle \geq 2^j (1-2^j \theta).
$$

Fix $\ell \in \mathbb N$ so that $2^{\ell} \theta \leq \frac12 \leq
2^{\ell+1} \theta$, yielding
\begin{equation}\label{eq:M}
\langle \varphi_{2^{\ell}}, f \rangle \geq \frac{1}{8\theta}.
\end{equation}
\medskip

Now, let $a_m = \langle \varphi_{2^m}, f \rangle$, and write, for
some $N \geq 1$,
$$
a_{\ell} - \frac{a_{N}}{2^{N-\ell}} = \sum_{m=\ell}^{N-1} \frac{2
a_m - a_{m+1}}{2^{m-\ell+1}}.
$$
By \eqref{eq:pwlimit}, we have $\langle P^i f, f \rangle \to 0$ as
$i \to \infty$, hence $\lim_{N \to \infty} \frac{a_N}{2^N} = 0$.
Using \eqref{eq:M} and taking $N \to \infty$ on both sides above
yields
$$
\frac{1}{8\theta} \leq a_{\ell} = \sum_{m=\ell}^{\infty} \frac{2 a_m
- a_{m+1}}{2^{m-\ell+1}}.
$$
Since $\sum_{m=\ell}^{\infty} \frac{1}{2^{m-\ell+1}}=1$, there must
exist some  $m \geq \ell$ with $2 a_m - a_{m+1} \geq
\frac{1}{8\theta}$. This establishes the claim~(\ref{eq:wantit}) for
$k=2^m$ and, in view of (\ref{eq:ratio1}),  completes the proof of
the lemma.
\end{proof}
\section{Balanced random walks with interaction}
First we start with some examples.
\subsection{Some examples}

\begin{example}
  A martingale $(X_n)_n$ in $\Z^2$, moves horizontally at times
  $t\in[2^{2k},2^{2k+1})$ with $k$ even and vertically
  $t\in[2^{2k+1},2^{2k+2})$ (to nearest neighbours, with equal
  probabilities in both cases).
\end{example}

Informally, this process is between one and two dimensional, as it has
long one-dimensional segments.

\begin{claim}
  This process is transient.
\end{claim}

\begin{proof}
  In the $k$th horizontal segment, the process can only visit $x$ if it is
  on the right horizontal line, which has probability
  $O(1/\sqrt{2^k})$. Since this is summable, the process only visits $x$ finitely many times. Similarly for vertical segments.
\end{proof}

\begin{example}[Benjamini--Kozma--Schapira \cite{BKS}]
  A martingale $(X_n)_n$ in $\Z^2$, moves vertically on the first visit to
  each site, and horizontally on subsequent visits.
\end{example}

\begin{question}
  Is this recurrent or transient?  \cite{BKS} includes this and several
  other open problems of similar nature.
\end{question}

\noindent\begin{minipage}{.7\linewidth}
\begin{example}[Nina Gantert; see Ofer Zeitouni's St.\ Flour lecture notes on RWRE]
  On $\Z^2$ again, a martingale moves horizontally with probability $2/3$
  (long arrows) and vertically with probability $1/3$ when $|x|<|y|$, and
  with opposite probabilities otherwise (including $|x|=|y|$).
\end{example}
\end{minipage}
\hfill
\begin{tikzpicture}[scale=.33]
  \path[use as bounding box] (-4,1)  rectangle (4.3, 3);
  \draw (-4,-4) -- (4,4);  \draw[-] (-4,4) -- (4,-4);
  \draw[<->] (3.5,-1) -- (3.5,1);   \draw[<->] (3,0) -- (4,0);
  \draw[<->] (-3.5,-1) -- (-3.5,1);   \draw[<->] (-3,0) -- (-4,0);
  \draw[<->] (-1,3.5) -- (1,3.5);   \draw[<->] (0,3) -- (0,4);
  \draw[<->] (-1,-3.5) -- (1,-3.5);   \draw[<->] (0,-3) -- (0,-4);
\end{tikzpicture}

\begin{proposition}\label{prop:eg3}
  This process is transient.
\end{proposition}

For the proof we use the following basic results.

\begin{lemma}\label{lem:transient_function}
  If a Markov chain on $S$ has non-constant $\phi:S\to\R^+$ with $P\phi\leq
  \phi$ (pointwise) then the chain is transient.
\end{lemma}

\begin{proof}
  $\phi(X_t)$ is a non-negative super-martingale, and so must converge,
  which contradicts recurrence.
\end{proof}

\begin{lemma}[Excessive measure]\label{lem:transient_measure}
  If $\mu P \leq \mu$ pointwise and $\mu P \neq \mu$ for a positive measure
  $\mu$ on $S$, then $(X_t)$ is transient.
\end{lemma}

\begin{proof}
  For any recurrent irreducible chain we have a stationary measure given by
  $\pi(x) = \E_a \sum_{i=0}^{\tau_a^+-1} 1_{X_i=x}$, where $\tau_a^+$ is
  the return time, and $a$ is an arbitrary reference state.  Consider the
  reverse chain with transitions $\hat p(x,y) =
  \frac{\pi(y)p(y,x)}{\pi(x)}$. Then $\pi$ is also stationary for $\hat P$.
  Moreover $P^n_{xx} = \hat P^n_{xx}$, and so $\hat P$ is also recurrent.

  In our case, the assumptions imply that $\phi = \frac\mu\pi$ has $\hat
  P\phi \leq \phi$.  By Lemma \ref{lem:transient_function} $\hat P$ is transient, and so
  $P$ must be transient as well.
\end{proof}

\begin{proof}[Proof of Proposition \ref{prop:eg3}]
  Consider $\mu\equiv 1$.  Then $\mu P\leq \mu$ and is strictly smaller at
  $0$.
\end{proof}

\subsection{Walks with few step distributions}

\cite{BKS} raise the following questions.

\begin{question}
  Fix two measures $\mu_1,\mu_2$ on $\Z^d$, $d\geq 3$ with mean $0$ and
  bounded support of full dimension.  Consider a process that makes steps
  with law $\mu_2$ on the first visit to a site, and $\mu_1$ on all
  subsequent visits.  When is this recurrent/transient?
\end{question}

\begin{question}
  More generally, what if the process moves from $X_t$ by $\mu_1$ or
  $\mu_2$ and the choice is adapted to $\CF_t$.
\end{question}
The next theorem answers these questions (from \cite{SPP12})
\begin{theorem}\label{thm:mainSPP}
  Fix any two measures $\mu_1,\mu_2$ on $\Z^d$, $d\geq 3$ with mean $0$ and
  bounded support of full dimension.  Let $(X_t)_t$ be a process such that
  conditioned on $\CF_t$ the step $X_{t+1}-X_t$ has law either $\mu_1$ or
  $\mu_2$.  Then $(X)$ is transient.
\end{theorem}

In contrast, there are recurrent processes with three possible step
distributions:

\begin{example}
  In $\Z^3$, make a step of $\pm1$ in the coordinate with maximal absolute
  value with probability $1-2\ve$, and in each of the other coordinates
  with probability $\ve$ each.
\end{example}

\begin{theorem}
  This process is recurrent for $\ve>0$ small enough.  In $\Z^d$ a similar
  construction works with $d$ measures.
\end{theorem}

Compare this to a continuous diffusion with larger variance in the radial
direction.  The absolute value is a Bessel process, and by adjusting the
covariance matrix, we can control the dimension and even make it less than
$2$, making the process recurrent.  The proof is based on careful
construction of a Lyapunov function.

\begin{proof}[Proof of Theorem \ref{thm:mainSPP}]
  First we investigate the case of a single increment measure $\mu$.
  Let $Z$ have law $\mu$, and consider $M = \Cov(\mu) = \Ev( Z Z^T)$.  By
  applying a linear map, we may assume this is a diagonal matrix
  $\diag(\lambda)$.

  Let $\phi(x) = |x|^{-2\alpha}$.  Using a Taylor expansion we have
  \[
  \frac{\phi(x+z)}{\phi(x)} = 1 - \frac{2\alpha x^T z}{|x|^2}
  - \frac{\alpha|z|^2}{|x|^2} + \frac{\alpha(\alpha+1)}{2}
  \frac{4x^Tzz^Tx}{|x|^4} + O(|x|^{-3}).
  \]
  Taking expectation (with $\Ev(Z)=0$) we get
  \begin{align*}
  \Ev\left( \frac{\phi(x+Z)}{\phi(x)}\right)
  &= 1 + \frac{\alpha}{|x|^4} \left(-\Ev |Z|^2 |x|^2 + 2(\alpha+1) x^T M x\right)
  + O(|x|^{-3}) \\
  &= 1 + \frac{\alpha}{|x|^4} \sum_i |x_i|^2 ( 2(\alpha+1) \lambda_i - \tr M)
  + O(|x|^{-3}) \\
  \end{align*}
  If
  \begin{equation}
    2\lambda_{\max} < \tr M     \label{eq:need}
  \end{equation}
  and $\alpha>0$ is sufficiently small then we get transience, since the
  sum is negative and dominates the error term.  We can truncate $\phi$ so
  that the inequality holds for small $x$ as well. Hence, transience follows from Lemma \ref{lem:transient_function}.

 Clearly \eqref{eq:need} is impossible for 2-dimensional matrices, so we need
  dimension at least 3.

  Note that if there are several increment laws $\mu_i$, the same $\phi$ may be
  super-harmonic for all of them simultaneously.  In that case, an
  arbitrary adapted choice of $\mu_i$ for the steps does not affect
  transience.

  For steps with a single law, we may consider instead the process
  $M^{-1/2} X$ which has $\Cov=I$, and \eqref{eq:need} holds.

  For a pair of matrices, we can always ensure $\eqref{eq:need}$, hence transience is guaranteed:

  \begin{claim}
    For any pair of $3\times3$ symmetric positive definite matrices
    $M_1,M_2$ there is an $A$ so that $A M_i A^T$ both satisfy
    \eqref{eq:need}.
  \end{claim}

  To see this, first apply some $A$ to make $M_1$ the identity, next
  diagonalize $M_2$ by a unitary matrix, (thus keeping $M_1=I$).  If at
  this point $M_2=\diag(a,b,c)$ apply $A=\diag(\sqrt{b/a},1,1)$ to finish,
  as the matrices are now $\diag(b/a,1,1)$ and $\diag(b,b,c)$.
\end{proof}

\noindent{\bf Acknowledgement.} We are grateful to Omer Angel,  Jian Ding and Miki Racz for scribing some of these notes, and to Lucas Boczkowski and Perla Sousi for helpful corrections.

\bibliographystyle{amsalpha}
\bibliography{spb}






\end{document}